\documentclass[11pt]{article}
\usepackage[utf8]{inputenc}
\usepackage[T1]{fontenc}
\usepackage{hyperref}
\usepackage{url}
\usepackage{booktabs}
\usepackage{nicefrac}
\usepackage{microtype}
\usepackage{lmodern}

\usepackage{footmisc}
\usepackage{graphicx}
\usepackage{subcaption}
\usepackage{amsfonts}
\usepackage{amsmath,amssymb,amsthm}
\usepackage{bbm}
\usepackage{xcolor}
\usepackage{fullpage}

\usepackage{enumitem}

\newcommand{\R}{\mathbb{R}}

\renewcommand{\P}{\mathbf{P}}

\newcommand{\E}{\mathbf{E}}
\newcommand{\di}{\mathrm{d}}

\newcommand{\wt}{\widetilde}
\newcommand{\wh}{\widehat}

\newcommand{\tr}{\mathrm{tr}}
\newcommand{\eps}{\epsilon}
\newcommand{\excessrisk}{\mathcal{E}}

\newcommand{\cond}{|}

\newcommand{\gaussdist}{\mathcal{N}}

\usepackage{xspace}
\newcommand{\ie}{\textit{i.e.}\@\xspace} 
\newcommand{\eg}{e.g.\@\xspace}
\newcommand{\iid}{i.i.d.\@\xspace}

\usepackage{algorithm,algpseudocode}

\DeclareMathOperator{\trace}{Tr}

\newtheorem{theorem}{Theorem}

\newtheorem{lemma}{Lemma}
\newtheorem{corollary}{Corollary}
\newtheorem{proposition}{Proposition}

\newtheorem{remark}{Remark}

\makeatletter
\let\@fnsymbol\@arabic
\makeatother

\title{Asymptotics of Ridge (less) Regression under General Source Condition}

\author{Dominic Richards\footnote{Department of Statistics,   University of Oxford, 24-29 St Giles', Oxford, OX1 3LB}
  \qquad
  Jaouad Mourtada\footnote{Department of Statistics, ENSAE, CREST, IP Paris, Palaiseau, France}
  \qquad
  Lorenzo Rosasco\footnote{MaLGa Center, Universitá degli Studi di Genova, Genova, Italy}\ %
  \footnote{Istituto Italiano di Tecnologia, Via Morego, 30, Genoa 16163, Italy}\ \footnote{Massachusetts Institute of Technology, Cambridge, MA 02139, USA}
}
\date{\today}

\begin{document}

\maketitle

\begin{abstract}
  We analyze the prediction error of ridge regression in an asymptotic regime where
  the sample size and dimension
  go to infinity at a proportional rate.
  In particular, we consider the role played by the structure of the true regression parameter.
  We observe that the case of a general deterministic parameter can be reduced to the case of a random parameter from a structured prior.
  The latter assumption is
  a natural adaptation of classic smoothness assumptions in nonparametric regression, which are known as source conditions in the the context of regularization theory for inverse problems. 
  Roughly speaking, we assume the large coefficients of the parameter are in correspondence to the principal components. 
  In this setting a precise characterisation of the test error is obtained, depending on the inputs covariance and regression parameter structure.
  We illustrate this characterisation in a simplified setting to investigate the influence of the true parameter on optimal regularisation for overparameterized models.
  We show that interpolation (no regularisation) can be optimal even with bounded signal-to-noise ratio (SNR), provided that the parameter coefficients are larger on high-variance directions of the data, corresponding to a more regular function than posited by the regularization term.
  This contrasts with previous work considering ridge regression with isotropic prior, in which case interpolation is only optimal in the limit of infinite SNR. 
\end{abstract}

\section{Introduction}

Understanding the generalisation properties of overparameterized model is a key question in machine learning, recently popularized by the study of neural  networks with millions and even  billions of parameters.
%
These models perform well in practice despite perfectly fitting (interpolating) the data, a property that seems at odds with classical statistical theory \cite{zhang2016understanding}.
This observation has lead to the investigation of the generalisation performance of methods that achieve zero training error (interpolators) \cite{liang2018just,belkin2018overfitting,belkin2019contradict,belkin2018understand,belkin2019two} and, in the context of linear least squares, the unique least norm solution to which gradient descent converges \cite{hastie2019surprises,bartlett2020benign,mitra2019understanding,belkin2019two,ghorbani2019linearized,muthukumar2020harmless,gerbelot2020asymptotic,nakkiran2020optimal}. 
Overparameterized linear models, where the number of variables exceed the number of points, 
are arguably the simplest and most natural setting where interpolation can be studied. 
Moreover, in some specific regimes,
neural networks can be approximated by suitable linear models 
\cite{jacot2018neural,du2018gradienta,du2018gradientb,allen2018convergence,chizat2019lazy}.

The learning curve (test error versus model capacity) for interpolators has been shown to  possibly exhibit a characteristic ``Double Descent'' \cite{advani2017high,belkin2019reconciling} shape, where the test error decreases after peaking at an ``interpolating'' threshold, that is, the model capacity required to interpolate the data. The regime beyond this threshold naturally captures the settings of neural networks \cite{zhang2016understanding}, and thus, has motivated its investigation \cite{mei2019generalization,spigler2019jamming,nakkiran2020optimal}. Indeed, for least squares regression, sharp characterisations  double descent  have been obtained for the least norm interpolating solution in the case of isotropic or auto-regressive covariates \cite{hastie2019surprises,belkin2019two} and random features \cite{mei2019generalization}.

For least squares regression the structure of the features and data can naturally influence performance. Within kernel regression (or inverse problems), for instance, it is often assumed that the parameter of interest is regular with respect to a given basis so as to to ensure a well-posed problem \cite{engl1996regularization,mathe2003geometry,bauer2007regularization}. Meanwhile for neural networks, inductive biases can be encoded in the network architecture \eg convolution layers for image classification \cite{lecun1989backpropagation,lecun1998gradient}. In each case, the problem is made easier by leveraging (through model design) that data encountered in practice exhibits lower dimensional structure owing to, for example, a set of simple physical laws governing the data generation. In contrast, the least squares models investigated beyond the interpolation threshold have focused on cases where the true regression parameter is isotropic 
\cite{dobriban2018high,hastie2019surprises}, which is a single instance in the range possible of alignements between the parameter and population covariance.
This has left open the natural questions of whether additional structure within the data generating distribution can be responsible for determining when interpolating is optimal.

In this work we investigate the performance of ridge regression, and its ridgeless limit, in a high dimensional asymptotic regime with a non-isotropic parameter.
We show that one can naturally reduce to a parameter sampled from a prior that, in short, encodes how the signal strength is distributed across the principal components of the covariates.
This structure has long been recognized as relevant in the statistics literature \cite{jolliffe1982note}, and is analogous to standard smoothness condition used within kernel regression and inverse problems, see e.g. \cite{engl1996regularization,mathe2003geometry,bauer2007regularization}. 

Specifically, a prior function encodes the parameter's norm when it is projected onto eigenspaces of the covariates population covariance.
Thus, it represents how aligned the ground truth is to the principle components in the data.
When considering the expected test error of ridge regression, this assumption can then encode \emph{any} deterministic parameter (Proposition \ref{prop:reduction-source-condition}).
Following the classic name in inverse problems, we call these assumptions {\em source conditions}.
 
Given this assumption, we then study the test error of ridge regression in a high-dimensional asymptotic regime when the number of samples and ambient dimension go to infinity in proportion to one another.  The limits of resulting quantities are then characterised by utilising tools from asymptotic Random Matrix Theory \cite{bai2010spectral,ledoit2011eigenvectors,dobriban2018high,hastie2019surprises}, with results specifically developed to characterise the influence of the prior function. This provides a 
natural and intuitive framework for studying the limiting test error of ridge regression,  characterised by the signal to noise ratio, regularisation, overparmeterisation, and now, the structure of the regression parameter as encoded by the source condition.

We then illustrate our general framework and results in a simplified setting that  highlights the role of model misspecification and its effect on prediction error and regularisation.
Specifically, we consider a population covariance with two types of eigenvectors: \emph{strong features}, associated with a common large eigenvalue (hence favored by the ridge estimator), as well as \emph{weak features}, with a common smaller eigenvalue.
This model is an idealization of a realistic structure for distributions, with some parts of the signal (associated for instance to high smoothness, or low-frequency components) easier to estimate than other, higher-frequency components.
The use of source conditions allows to study situations where the true coefficients
are either more or less aligned with the principal components,
than implicitly postulated by the ridge estimator,
a form of model misspecification which affects predictive performance.
This encodes the difficulty of the problem, and allows to distinguish between ``easy'' and ``hard'' learning problems.   
We now summarise this work's primary contributions.
\\
\begin{itemize}[leftmargin = *]
\item \textbf{Asymptotic prediction error under general source condition.} An asymptotic characterisation of the test error under a general source condition on the regression parameter is provided. This required characterizing the limit of certain trace quantities, and provides a natural framework for investigating the performance of ridge regression. (Theorem \ref{main:thm})
\item \textbf{Interpolating can be optimal even in noisy cases.}
  In the overparameterised regime, we show that interpolation can lead to smaller risk than any positive choice of the regularisation parameter. This occurs in the favorable situation where the regression parameter is larger in high-variance directions of the data, and the signal-to-noise ratio is large enough (but finite).
  Previously, for least squares regression with isotropic prior, the optimal regularisation choice was zero only in the limit of infinite signal to noise ratio \cite{dicker2016ridge,dobriban2018high}. (Section \ref{sec:TwoBulks:ZeroRegOpt})
\end{itemize}
Our analysis of the strong and weak features model also provides asymptotic characterisations of a number of phenomena recently observed within the literature. That is, augmenting the data by adding noisy co-ordinates performs implicit regularisation and can recover the performance of optimally tuned regression restricted to the strong features \cite{kobak2018implicit}. Also, we show an additional peak occurring in the learning curve beyond the interpolation threshold for the ridgeless bias and variance \cite{nakkiran2020optimal}. These insights are presented in Sections \ref{sec:TwoBulks:NoisyBulkReg} and  \ref{sec:TwoBulks:Ridgeless:Variance}, respectively.  

The remainder of this work is organized as follows. Section \ref{sec:RelatedLit} covers the related literature. Section \ref{sec:SetupandThm} describes the setting, and provides the general theorem. Section \ref{sec:TwoBulksModel} formally introduces the strong and weak features model, and presents the aforementioned insights. Section \ref{sec:conc} gives the conclusion. 

\subsection{Related Literature}
\label{sec:RelatedLit}
Due to the large number of works investigating interpolating methods as well as double descent, we next  focus on works that consider the asymptotic regime. 

\paragraph{High-Dimensional Statistics.}
Random matrix theory has found numerous applications in high-dimensional statistics \cite{yao2015sample,elkaroui2018random}.
In particular, asymptotic random matrix theory has been leveraged to study the predictive performance of ridge regression under a well-specified linear model with an isotropic prior on the parameter, for identity population covariance \cite{karoui2011geometric,karoui2013asymptotic,dicker2016ridge,tulino2004random} and then general population covariance~\cite{dobriban2018high}.  More recently, \cite{mahdaviyeh2019asymptotic} considered the limiting test error of the least norm predictor under the spiked covariance model \cite{johnstone2001distribution} where both a subset of eigenvalues and the ratio of dimension to samples diverge to infinity.
They show the bias is bounded by the norm of the ground truth projected on the eigenvectors associated to the subset of large eigenvalues.
In contrast, our work follows standard assumption in kernel regression or inverse problems literature \cite{engl1996regularization,mathe2003geometry,bauer2007regularization}, by adding structural assumptions on the parameter through the variation of its coefficients along the covariance basis.
Finally, we note the works \cite{liao2019inner,liao2019large}
that utilise tools from random matrix theory to characterise the prediction performance of linear estimators in the context of classification.

\paragraph{Double Descent for Least Squares.}
While interpolating predictors (which perfectly fit training data), are classically expected to be sensitive to noise and exhibit poor out-of-sample performance, empirical observations about the behaviour of artificial neural networks \cite{zhang2016understanding} challenged this received wisdom.
This surprising phenomenon, where interpolators can generalize, has first been shown for some local averaging estimators~\cite{belkin2019contradict,belkin2018overfitting}, kernel ``ridgeless'' regression \cite{liang2018just}, and linear regression, where \cite{bartlett2020benign} characterised the variance of the ridgeless estimator up to universal constants.
A ``double descent'' phenomenon for interpolating predictors, where test error can decrease past the interpolation threshold, has been suggested by~\cite{belkin2019reconciling}.

This double descent curve has motivated a number of works  established in the context of asymptotic least squares \cite{hastie2019surprises,mei2019generalization,belkin2019two,xu2019number,gerbelot2020asymptotic,muthukumar2020harmless,nakkiran2020optimal}.
The work \cite{hastie2019surprises} considers either isotropic or auto-regressive features, while \cite{louart2018random,mei2019generalization}
consider Random Features constructed from a non-linear functional applied to the product of isotropic covariates and a random matrix. In \cite{hastie2019surprises,mei2019generalization} the data is assumed to be generated with an isotropic ground truth with some model mis-specification.
The works \cite{mitra2019understanding,gerbelot2020asymptotic,muthukumar2020harmless} considers recovery guarantees under sparsity assumptions on the parameter, with \cite{gerbelot2020asymptotic} showing a peak in the test error when the number of samples equals the sparsity of the true predictor.
The work \cite{muthukumar2020harmless} considers recovery properties of interpolators in the non-asymptotic regime. In contrast to these works, we consider structural assumption on the ground truth in terms of the population covariance that directly follow from standard smoothness conditions in the kernel regression/ inverse problem literature. 

The work \cite{nakkiran2020optimal} gave empirical evidence showing additional peaks in the test error can occur beyond the interpolation threshold when the covariance and ground truth parameter are misaligned.
These empirical observations are verified by the theory in this paper.
Along these lines, we also note the concurrent work \cite{chen2020multiple} which shows a variety of different learning curves are possible for interpolating least squares regression when the sample size is fixed and dimension of the problem is varied. 

\paragraph{Concurrent Work.}
We now review independent work, which appeared in parallel to 
or since the first version of this paper.
The works \cite{wu2020optimal, amari2020does} also considers the asymptotic prediction performance of ridge regression with prior assumptions on the parameter.
Similar to us, \cite{wu2020optimal} shows that interpolating is optimal when the parameter is sufficiently ``aligned'' to the population covariance and the signal to noise ratio is large.
Our technical formulations are formally different but related:
they express the alignment between the parameter and the population covariance in terms of the projections of $\beta$ on the eigenvectors of $\Sigma$, whereas we encode it through the source function $\Phi$; the correspondence between the two formulations is obtained through Proposition~\ref{prop:reduction-source-condition}.
They also include additional study of the sign of optimal ridge penalty.
Meanwhile, \cite{hastie2019surprises} has been recently updated to include refined non-asymptotic results that build upon both our work and \cite{wu2020optimal},
also accounting for the structure of the regression parameter along principal directions.
They derived a general non-asymptotic bound, controlling the difference between the finite-sample risk and its high-dimensional limit.

\section{Dense Regression with General Source Condition
}
\label{sec:SetupandThm}
In this section we formally introduce the setting as well as the main theorem. Section \ref{sec:setup:data} introduces the linear regression setting. Section \ref{sec:ReductionSourceCondition} shows the prior assumption we consider can encapsulate a general ground truth predictor. 
Section \ref{sec:setup:RMT} introduces the functionals that arise from asymptotic random matrix theory. Section \ref{sec:MainTheorem} presents the main theorem. 

\subsection{Problem Setting}
\label{sec:setup:data}

We start by introducing the linear regression setting and the general source condition.

\paragraph{Linear Regression.}
We consider prediction in a random-design linear regression setting with Gaussian covariates.
Let $\beta^\star \in \R^d$ denote the true regression parameter, $\Sigma \in \R^{d \times d}$ the population covariance, and $\sigma^2 > 0$ the noise variance.
We consider an \iid dataset $\{ (x_i, y_i) \}_{1 \leq i \leq n}$ such that for $i = 1, \dots, n$,
\begin{equation}
  \label{eq:LinearModel}
  y_i = \langle \beta^\star, x_i\rangle + \sigma \eps_i ,
  \qquad x_i \sim \gaussdist (0, \Sigma) ,
\end{equation}
and the noise satisfies $\E [ \eps_i \cond x_i ] = 0$, $\E [ \eps_i^2 \cond x_i ] = 1$. In what follows, let $Y = (y_1,\dots,y_n), \epsilon= (\epsilon_1,\dots,\epsilon_n) \in \R^{n}$, and the design matrix $X \in \R^{n \times d}$.
Given the $n$ samples the objective is to derive an estimator $\beta \in \R^d$ that minimises the error of predicting a new response. For a fixed parameter $\beta^\star$, the test risk is then
$
    R(\beta) 
    = \E[ (\langle x, \beta \rangle - y )^2 ]
    = \|\Sigma^{1/2}(\beta - \beta^{\star}) \|_2^2 
    + \sigma^2
$,
where the expectation is with respect to a new response sampled according to~\eqref{eq:LinearModel}. We consider ridge regression~\cite{hoerl1970ridge,tikhonov1963regularization}, defined for $\lambda > 0$ by 
\begin{equation}
  \label{eq:def-ridge}
  \wh \beta_{\lambda} := \Big( \frac{X^{\top} X}{n} + \lambda I \Big)^{-1} \frac{X^{\top} Y}{n}.
\end{equation}

\paragraph{Source Condition.}
We consider an average-case analysis where the parameter $\beta^{\star}$ is random, sampled with covariance encoded by a \emph{source function} $\Phi: \R^+ \to \R^+$, which describes how coefficients of $\beta^\star$ vary along eigenvectors of $\Sigma$.
Specifically, denote by $\{(\tau_j,v_j)\}_{1\leq j \leq d}$ the eigenvalue-eigenvector pairs of $\Sigma$, ordered so that $\tau_1 \geq \tau_2 \geq \dots \geq \tau_d \geq 0$, and let $\Phi(\Sigma)=  \sum_{i=1}^{d} \Phi(\tau_i) v_i v_i^{\top}$.
For $r  > 0$ 
the parameter $\beta^{\star}$ is such that
\begin{align}
\label{equ:BetaPrior}
\E[\beta^{\star}] = 0, \quad\quad \E[ \beta^{\star}(\beta^{\star})^{\top}] =  \frac{r^2}{d} \Phi(\Sigma).
\end{align}

For estimators linear in $Y$ (such as ridge regression), the expected risk only depends on the first two moments of the prior on $\beta^\star$, hence one can assume a Gaussian prior $\beta^\star \sim \gaussdist (0, r^2\Phi(\Sigma)/d )$.
Under prior~\eqref{equ:BetaPrior}, $\Phi (\Sigma)^{-1/2} \beta^\star$ has isotropic covariance $I / d$, so that $\E \| \Phi (\Sigma)^{-1/2} \beta^\star \|^2 = 1$.
This means that the coordinate $\beta_j := \langle \beta^\star, v_j\rangle$ of $\beta^\star$ in the $j$-th direction has standard deviation $\sqrt{\Phi (\tau_j) / d}$.
We note that, as $d \to \infty$, $\beta^\star$ has a ``dense'' high-dimensional structure, where the number of its components grows with $d$, while their magnitude decreases proportionally.
This prior is an average-case, high-dimensional analogue of the standard \emph{source condition} considered in inverse problems and nonparametric regression~\cite{mathe2003geometry,bauer2007regularization}, which describes the behaviour of coefficients of $\beta^\star$ along the eigenvector basis of $\Sigma$.
In the special case $\Phi (x) = x^{\alpha}$, $\alpha\ge 0$, one has $\E \| \Sigma^{-\alpha/2} \beta^\star \|^2 = r^2$.
For a Gaussian prior,  $\Sigma^{-\alpha/2} \beta^\star \sim \gaussdist (0, r^2 I/d )$, which
is rotation invariant with squared norm distributed as $r^2\chi_d^2 / d$ (converging to $r^2$ as $d \to \infty$), hence ``close'' to the uniform distribution on the sphere of radius $r$. In Section \ref{sec:ReductionSourceCondition} we show, when considering the expected test error, that this source assumption can then encode any deterministic ground truth parameter.

\paragraph{Easy and Hard Problems.}
The case of a constant function $\Phi (x) \equiv 1$ corresponds to an isotropic prior under the Euclidean norm used for regularisation, and has been studied by \cite{dicker2016ridge,dobriban2018high,hastie2019surprises}.
In this case (see Remark~\ref{rem:oracle-estimator} below), properly-tuned ridge regression (in terms of $r^2$) is optimal in terms of average risk.
The influence of $\Phi$ can be understood in terms of the average signal strength in eigen-directions of $\Sigma$.
Specifically, let $v_j$ be an eigenvector of $\Sigma$, with associated eigenvalue $\tau_j$.
Then, given $\beta^\star$, the signal strength in direction $v_j$ (namely, the contribution of this direction to the signal) is $\E_x \langle \langle \beta^\star, v_j\rangle v_j, x \rangle^2 = \tau_j \langle \beta^\star, v_j\rangle^2$, and  its expectation over $\beta^\star$ is $\tau_j \Phi (\tau_j)$.
When $\Phi$ is increasing, strength along direction
$v_j$
decays faster as $\tau_j$ decreases,  than postulated by the ridge regression penalty.
In this sense, the problem is lower-dimensional, and hence ``easier'' than for constant $\Phi$; likewise, a decreasing $\Phi$ is associated to a slower decay of coefficients, and therefore a ``harder'', higher-dimensional problem.
While our results do not require this restriction, it is natural to consider functions $\Phi$ such that $\tau \mapsto \tau \Phi (\tau)$ is non-decreasing, so that principal components (with larger eigenvalue) carry more signal on average; otherwise, the norm used by the ridge estimator favours the wrong directions.
In this respect, the hardest prior is obtained for $\Phi (\tau) = \tau^{-1}$, corresponding to the isotropic prior in the prediction norm induced by $\Sigma$: for this un-informative prior, all directions have same signal strength.
Finally, note that in the standard nonparametric setting of reproducing kernel Hilbert spaces, source conditions are related to smoothness of the regression function~\cite{steinwart2009optimal}.

\begin{remark}[Oracle estimator]
  \label{rem:oracle-estimator}
  The best linear (in $Y$) estimator in terms of average risk can be described explicitly.
  It corresponds to the Bayes-optimal estimator under prior $\gaussdist (0, r^2 \Phi(\Sigma) / d)$ on $\beta^\star$, which writes:
  \begin{equation}
    \label{eq:oracle-estimator}
    \wt \beta
    = \Big( \frac{X^\top X}{n} + \frac{\sigma^2}{r^2} \frac{d}{n} \Phi (\Sigma)^{-1} \Big)^{-1} \frac{X^\top Y}{n}.
  \end{equation}
  This estimator requires knowledge of $\Sigma$ and $r^2\Phi$.
  In the special case of an isotropic prior with $\Phi \equiv 1$, the oracle estimator is the ridge estimator~\eqref{eq:def-ridge} with $\lambda = (\sigma^2 d) / (r^2 n)$.
\end{remark}

\subsection{Reduction  to Source Condition} 
\label{sec:ReductionSourceCondition}

In this section, we show that the source condition introduced in Section~\ref{sec:setup:data} is not restrictive, since the general case reduces to it.
Specifically, the following proposition shows that the expected error for \emph{any} deterministic $\beta^\star \in \R^d$ is equal to the averaged error according to a prior with covariance of the form $\Phi (\Sigma)$ for some function $\Phi = \Phi_{\beta^\star,\Sigma}$ depending on $\beta^\star$ and $\Sigma$. 
\begin{proposition}[Reduction to source condition]
  \label{prop:reduction-source-condition}
  Consider data generated according to \eqref{eq:LinearModel}.
  Let $\beta^\star \in \R^d$, $\beta_j = \langle \beta^\star, v_j\rangle$ for $j=1,\dots,d$ and $\Phi = \Phi_{\beta^\star,\Sigma} : \R^+ \to \R$ be a function such that, for $\tau \in \{ \tau_1, \dots, \tau_d \}$, 
  \begin{equation}
    \label{eq:equivalent-source}
    \Phi (\tau)
    = \frac{d}{|J(\tau)|} \sum_{j \in J (\tau)} \beta_j^2
    \, ,
  \end{equation}
  where $J (\tau) = \{ 1 \leq j \leq d : \tau_j = \tau \}$.
  Let $\Pi$ be a distribution on $\R^d$ such that $\E_{\beta\sim \Pi} [ \beta ] = 0$ and $\E_{\beta\sim \Pi} [ \beta \beta^\top ] = \Phi (\Sigma) / d$.
  Then, we have
  \begin{equation*}
    \E_{X,\eps} \big[  \|\Sigma^{1/2}(\widehat{\beta}_{\lambda} - \beta^{\star})\|_2^2\big]
    \! = \! \E_{\beta \sim \Pi} \E_{X,\eps} \! \big[  \|\Sigma^{1/2}(\widehat{\beta}_{\lambda} - \beta)\|_2^2\big]
.
  \end{equation*}
\end{proposition}

The equality in Proposition~\ref{prop:reduction-source-condition} holds for finite samples and deterministic $\beta^\star$ (and $\Sigma$), and provides a reduction to the setting of random $\beta^\star$ used in remaining sections.

On a technical side, the equality in Proposition \ref{prop:reduction-source-condition} holds for the expected test error, while the remaining results within this work align with prior work~\cite{dobriban2018high} where expectation is taken with respect to the parameter and noise only (conditionally on covariates $X$) i.e. $\E_{\epsilon,\beta^{\star}}[R(\wh \beta_{\lambda}) - R(\beta^{\star})] = \E_{\epsilon,\beta^{\star}}[\|\Sigma^{1/2}(\beta - \beta^{\star}) \|_2^2]$.
Note that convergence results on the conditional risk can be integrated under suitable domination assumptions, for instance with positive ridge parameter $\lambda$.
In addition, framing our next convergence results in the context of deterministic $\beta^\star$ would lead to consider source functions $\Phi_{\beta^\star, \Sigma} = \Phi_d$ depending on the dimension $d$, and converging to a fixed function $\Phi$ in a suitable sense as $d \to \infty$.
For the sake of simplicity, we instead work in the setting of random parameter $\beta^\star$ with a fixed source function $\Phi$.

On another note, the generalised ridge estimator, which penalises with respect to a general covariance $\|P\beta\|_2^2$ for a positive definite matrix $P$, reduces after rescaling to standard ridge regression with an appropriate prior and covariate covariance.
Namely, the problem instance with prior, penalisation and covariate covariances $(\Pi,\! P,\! \Sigma)$ is equivalent to using $(P^{1/2} \Pi P^{1/2}, I,  P^{-1/2}\Sigma P^{-1/2})$ with parameterisation $\widetilde{\beta}^{\star}  =  P^{1/2}\beta^{\star}$, $\widetilde{\beta}  =  P^{1/2} \beta$ and $\widetilde{X}  =  X P^{-1/2}$.

\subsection{Random Matrix Theory}
\label{sec:setup:RMT}

Let us now describe the considered asymptotic regime, as well as quantities and notions from random matrix theory that appear in the analysis.

\paragraph{High-Dimensional Asymptotics.}
We study the performance of the ridge estimator $\wh \beta_{\lambda}$ under high-dimensional asymptotics \cite{karoui2011geometric,karoui2013asymptotic,dicker2016ridge,dobriban2018high,tulino2004random,bai2010spectral}, where the number of samples and dimension go to infinity $n,d \rightarrow \infty$ proportionally with $d/n \to \gamma \in (0,\infty)$.
This setting enables precise characterisation of the risk, beyond the classical regime where $n\to \infty$
with fixed true distribution.

The ratio $\gamma = d/ n$ plays a key role.
A value of $\gamma > 1$ corresponds to an overparameterised model, with more parameters than samples.
Some care is required in interpreting this quantity:
indeed, for a fixed sample size $n$, varying $\gamma$ changes $d$ and hence the underlying distribution.
Hence, $\gamma$ should not be interpreted as a degree of overparmeterisation. Rather, it quantifies the sample size relatively to the dimension of the problem.

\paragraph{Random Matrix Theory.}
Following standard assumptions \cite{ledoit2011eigenvectors,dobriban2018high}, assume the spectral distribution of the covariance $\Sigma$ converges almost surely to a probability distribution $H$ supported on $[h_1,h_2]$ for $0 < \! h_1 \! \leq \! h_2 < \! \infty$. Specifically, denoting the cumulative distribution function of the population covariance eigenvalues as $H_{d}(\tau) \! = \! \frac{1}{d} \! \sum_{i=1}^{d} \mathbbm{1}(\tau)_{[\tau_i,\infty)}$, we have $H_d (\tau) \rightarrow H(\tau)$ almost surely as $d \!  \rightarrow \! \infty$.

A key quantity utilised within the analysis is the \emph{Stieltjes Transform} of the empirical spectral distribution, defined for $z \in \mathbb{C}\backslash \R_{+}$ as $\widetilde{m}(z) := d^{-1} \trace\big( \big( \frac{X^{\top}X}{n} - zI  \big)^{-1}\big)$.
Under appropriate assumptions of the covariates $x$ (see for instance \cite{dobriban2018high}) it is known as $n,d\rightarrow \infty$ the  Stieltjes Transform of the empirical covariance $\widetilde{m}(z)$ converges almost surely to a Stieltjes transform $m(z)$ that satisfies the following stationary point equation
\begin{align}
\label{equ:STransform}
    m(z) = \int_{0}^{\infty} \frac{1}{\tau (1 - \gamma(1+z m(z))) - z} \di H(\tau).
\end{align}
For an isotropic covariance $\Sigma = I$ the limiting spectral distribution is a point mass at one, and the above equation can be solved for $m(z)$ where it is the Stieltjes Transform of the Marchenko-Pastur distribution \cite{marvcenko1967distribution}. 
For general spectral densities, the stationary point equation \eqref{equ:STransform} may not be easily solved algebraically, but still yields insights into the limiting properties of quantities that arise. One tool that we will use to gain insights will be the \textit{companion transform} $v(z)$ which is the Stieltjes transform of the limiting spectral distribution of the Gram matrix $n^{-1} X X^{\top}$. It is related to $m(z)$ through the following equality 
$
    \gamma(m(z) + 1/z) = v(z) + 1/z 
    \text{  for all  } 
    z \in \mathbb{C} \backslash \R_{+}.
$
Finally, introduce the $\Phi$-weighted Stieltjes Transform defined for $z \in \mathbb{C} \backslash \R_{+}$
\begin{align*}
    \Theta^{\Phi}(z) : = 
    \int \Phi(\tau) \frac{1}{\tau (1 - \gamma(1+z m(z))) - z} \di H(\tau),
\end{align*}
which is the limit of the trace quantity $d^{-1}\trace\big( \Phi(\Sigma) (\frac{X^{\top} X}{n} - z I )^{-1} \big)$
\cite{ledoit2011eigenvectors}.

\subsection{Main Theorem: Asymptotic Risk under General Source Condition}
\label{sec:MainTheorem}
Let us now state the main theorem of this work, which provides the limit of the ridge regression risk.
\begin{theorem}
\label{main:thm}
Consider the setting described in Section \ref{sec:setup:data} and \ref{sec:setup:RMT}. Suppose $\Phi$ is a real-valued bounded function defined on $[h_1,h_2]$ with finitely many points of discontinuity and let $v^{\prime}(z) = \partial v(z)/\partial z$. If $n,d \rightarrow \infty$ with  $\gamma = d/n  \in (0,\infty)$  then almost surely
$
    \E_{\epsilon,\beta^{\star}}[ R(\wh \beta_{\lambda})-
    R(\beta^{\star})]
     \rightarrow
    R_{\mathrm{Asym}}(\lambda)
$ where 
    \begin{align*}
      R_{\mathrm{Asym}}(\lambda)  =  
    \underbrace{  
    \sigma^2 \Big(  \frac{ v^{\prime}(-\lambda) }{v(-\lambda)^2}  -  1 \Big)}_{\text{Variance}}
     + 
     \underbrace{ 
     r^2\frac{ \Theta^{\Phi}(-\lambda)  +  \lambda \frac{\partial \Theta^{\Phi}(-\lambda)}{\partial \lambda}}{v(-\lambda)^2} 
     }_{\text{Bias}}.
\end{align*}
\end{theorem}
The above theorem characterises the expected test error of the ridge estimator when the sample size and dimension go to infinity $n,d \rightarrow \infty$ with $d/n = \gamma \in (0,\infty)$, and $\beta^{\star}$ is distributed as \eqref{equ:BetaPrior}.
The asymptotic risk in Theorem \ref{main:thm} is characterised by the relative sample size $\gamma$, the limiting spectral distribution $H$, and the source function $\Phi$ (normalising $\sigma^2 = r^2 =1$).
This provides a general form for studying the asymptotic test error for ridge regression in a dense high-dimensional setting.
The source condition affects the limiting bias; to evaluate it we are required to study the limit of the trace quantity $d^{-1} \trace\big(\Sigma (\frac{X^{\top} X}{n} - zI)^{-1} \Phi(\Sigma)(\frac{X^{\top} X}{n} - zI)^{-1} \big)$, which is achieved utilising techniques from both \cite{chen2011regularized} and \cite{ledoit2011eigenvectors} (key steps in proof of Lemma \ref{lem:Limits} Appendix \ref{app:proof-lemma-rmt}). The variance term in Theorem \ref{main:thm} aligns with that seen previously in \cite{dobriban2018high},  as the structure of $\beta^{\star}$ only influences the bias. 

We now give some examples of asymptotic expected risk in Theorem \ref{main:thm} for $3$ different structures of $\beta^{\star}$, namely $\Phi(x) = 1$ (isotropic), $\Phi(x) = x$ (easier case) and $\Phi(x)= x^{-1}$ (harder case). 
\begin{corollary}
\label{Cor:SpikedCovar}
Consider the setting of Theorem \ref{main:thm}. If $n,d\rightarrow \infty$ with $\gamma = d/n$, then almost surely
\begin{align*}
     \E_{\epsilon,\beta^{\star}}\![ R(\wh \beta_{\lambda})
     -  
    R(\beta^{\star})]
     \rightarrow \!
    \sigma^2 \Big( \! \frac{  v^{\prime}(-\lambda) }{(v(-\lambda))^2}\! - \! 1 \! \Big) 
     \, + 
    r^2 \begin{cases}
    \frac{v^{\prime}(-\lambda) }{\gamma  (v(-\lambda))^4} 
    -  \frac{1}{\gamma  v(-\lambda)^2}
     & \text{if } \Phi(x) = x \\
     \frac{1 }{\gamma v(-\lambda) } \! - \! \frac{\lambda}{\gamma} 
     \frac{v^{\prime}(-\lambda)}{ ( v(-\lambda))^2} 
     & \text{if } \Phi(x) = 1 \\
     \frac{2 \lambda }{\gamma}  \frac{v^{\prime}(-\lambda)}{ v(-\lambda)}
      +  (1 - \frac{1}{\gamma} \big) \frac{v^{\prime}(-\lambda)}{ v(-\lambda)^2}  
     - \frac{1}{\gamma  } 
    &  \text{if } \Phi(x) = 1/x
    \end{cases}
\end{align*}
\end{corollary}
The three choices of source function $\Phi$ in Corollary \ref{Cor:SpikedCovar} are cases where the asymptotic bias in Theorem \ref{main:thm} can be expressed in terms of the companion transform and its first derivative. The expression in the case $\Phi(x) \! = \! 1$ was previously investigated in \cite{dobriban2018high}, while for $\Phi(x) \! = \! x$ the bias aligns with quantities previously studied in \cite{chen2011regularized}, and thus, can be simply plugged in. For $\Phi(x) \! = \! x^{-1} $, algebraic manipulations similar to the $\Phi(x)\! = \! x$ case allow  $\Theta^{\Phi}(z)$ to be simplified. Finally, for $\Phi(x) \! =  \! 1$ it is clear how the bias and variance can be brought together and simplified yielding optimal regularisation choice $\lambda \! = \! \sigma^2 \gamma/r^2$ \cite{dobriban2018high}, see also Remark~\ref{rem:oracle-estimator}.
As noted in Section~\ref{sec:setup:data}, $\Phi (x) \!=  \! x^{-1}$ corresponds to a ``harder" case, with no favoured direction, while $\Phi (x) \! = \! x$ corresponds to an ``easier'' case with faster coefficient decay.

\section{Strong and Weak Features Model}
\label{sec:TwoBulksModel}

In this section we consider a particular covariance structure, 
the \emph{strong and weak features model}.
Let $U_1 \in \R^{d_1 \times d}$ and $U_2 \in \R^{d_2 \times d}$ be orthonormal matrices whose rows forms an orthonormal basis of $\R^d$ and $d_1 + d_2 = d$. The covariance is then for $\rho_1 \geq \rho_2 > 0$ 
\begin{align}
\label{equ:cov:Bulk}
    \Sigma = \rho_1 U_1^{\top}U_1 + \rho_2 U_2^{\top} U_2. 
\end{align}
We call elements of the span of rows of $U_1$ \emph{strong features}, as they are associated to the dominant eigenvalue $\rho_1$.
Similarly, $U_2$ is associated to the \emph{weak features}.
The size of $U_1,U_2$ go to infinity $d_1,d_2 \! \rightarrow \! \infty$ with the sample size $n \! \rightarrow \! \infty$, with $d_i/d \! \rightarrow \! \psi_i \in (0,1)$ and thus $\psi_1 \! + \! \psi_2 \! = \! 1$.
The limiting population spectral measure is then atomic $dH(\tau) \! = \! \psi_1 \delta_{\rho_1} \! +\! \psi_2 \delta_{\rho_2}$.

The parameter $\beta^{\star}$  has covariance 
$\E[\beta^{\star} (\beta^{\star})^{\top}] = \frac{r^2}{d} \big( \phi_1 U_1^{\top} U_1 \! + \! \phi_2 U_2^{\top} U_2 \big)$, where $\phi_1 ,\phi_2$ are the coefficients for each type of feature and the source condition is $\Phi(x) \! = \! \phi_1 \mathbbm{1}_{x=\rho_1} \! + \! \phi_2 \mathbbm{1}_{x=\rho_2}$. 
The coefficients $\phi_1,\phi_2$ encode the composition of the ground truth in terms of strong and weak features, and thus, the difficulty of the estimation problem.
The case $\phi_1 \! = \! \phi_2$ corresponds to the isotropic prior, while the case $\phi_1 > \phi_2$ corresponds to faster decay and hence an ``easier'' problem.
Specifically, if $\phi_1>\phi_2$ increases, $\beta^\star$ has faster decay, 
the problem becomes ``easier'' since the ground truth is increasingly made of strong features. Therefore, if $\phi_1/\phi_2 \! \geq \! 1$ we say the problem is \emph{easy}, while if $\phi_1/\phi_2 \! < \! 1$ we say the problem is \emph{hard}.

Under the model just introduced, Theorem \ref{main:thm} provides the following asymptotic characterization for the expected test risk as $n,d \rightarrow \infty$
\begin{align*}
R_{\mathrm{Asym}}(\lambda) =  
\frac{  v^{\prime}(-\lambda) }{v(-\lambda)^2}
\Big( \! \sigma^2 \!\! + \! r^2 \sum_{i=1}^{2} \frac{\phi_i \psi_i \rho_i }{ (\rho_i v(-\lambda) \! + \! 1)^2} 
\! \Big)
\! - \!\sigma^2
.
\end{align*}
To gain insights into the performance of least squares when data is generated from the strong and weak features model,
we now investigate the above limit in the overparameterised setting $\gamma > 1$.
The insights are summarised in the following sections.  Section \ref{sec:TwoBulks:ZeroRegOpt} shows that zero regularisation is optimal for easy problems with high signal to noise ratio.
Section \ref{sec:TwoBulks:NoisyBulkReg} shows how weak features can be used as a form of regularisation similar to ridge regression.
Section \ref{sec:TwoBulks:Ridgeless:Variance} present findings related to the ridgeless bias and variance. 

\paragraph{Source Condition Reduction for Strong and Weak Features Model.}
Following Section \ref{sec:ReductionSourceCondition}, the case of a general deterministic parameter $\beta^{\star} \in \mathbb{R}^{d}$ can be encoded by the strong and weak features setting just described. Namely,  let $\beta_1, \beta_2$ be the respective projections of $\beta^\star$ onto  rows of $U_1, U_2$.
Then $\phi_1 = \| \beta_1 \|^2_2 d /d_1 = $ and $\phi_2 = \| \beta_2 \|^2_2 d /d_2$, and thus, the coefficients $\phi_1,\phi_2$ align with the norm of the ground truth $\beta^{\star}$ projected on each bulk.

\subsection{Interpolating can be optimal in the presence of noise
}
\label{sec:TwoBulks:ZeroRegOpt}
In this section, we investigate how the true regression function, namely the parameter $\beta^\star$ (through the source condition), affects optimal ridge regularisation. 
We begin with the following corollary, which, in short, describes when zero regularisation can be optimal. Let us denote the derivative of the asymptotic risk $R_{\mathrm{Asym}}(\lambda)$ with respect to the regularisation $\lambda$ as $R^{\prime}_{\mathrm{Asym}}(\lambda)$. 
\begin{corollary}
\label{Cor:PeciseOptimalReg}
Consider the strong and weak features model with $\gamma = 2$, $\psi_1 = \psi_2 = 1/2$ and $\E[\|\beta^{\star}\|_2^2] = r^2 $. If 
\begin{align*}
 \underbrace{ \frac{r^2}{\sigma^2} \frac{\rho_1 \rho_2}{(\sqrt{\rho_1} + \sqrt{\rho_2})^2} }_{\text{Signal to Noise Ratio}}
\Big( \underbrace{ \frac{\phi_1 \sqrt{\rho_1} + \phi_2 \sqrt{\rho_2}}{\sqrt{\rho_1} + \sqrt{\rho_2}} - 1  }_{\text{Alignment}}\Big)  \geq  1
\end{align*}
then $R_{\mathrm{Asym}}^{\prime}(0) \geq 0$. Otherwise $R_{\mathrm{Asym}}^{\prime}(0) < 0$.
\end{corollary}
Corollary \ref{Cor:PeciseOptimalReg} states that if the ground truth is aligned $\phi_1 \! > \! 1$ (the signal concentrates more on strong features) and the signal to noise ratio $r^2/\sigma^2$ is sufficiently large, then the derivative of the asymptotic test error at zero regularisation is positive $R^{\prime}_{\mathrm{Asym}}(0) \! \geq \! 0$. The interpretation being that interpolating is optimal if adding regularisation increases (locally at $0$) the test error. The case $\gamma \! = \! 2$ and $\psi_1 \! = \! \psi_2 \! = \! 1/2$ is considered, as the companion transform at zero takes a simple form  $v(0) \! = \! 1/\sqrt{\rho_1\rho_2}$, allowing the derivatives $v^{\prime}(0)$ and $v^{\prime\prime}(0)$, and thus $R^{\prime}_{\mathrm{Asym}}(0)$, to be tractable. 

Looking to Figure \ref{fig:OptReg2} plots for  the performance of optimally tuned ridge regression (\textit{Left}) and the optimal choice of regularisation parameter (\textit{Right}) against (a monotonic transform) of the eigenvalue ratio $\rho_1/\rho_2$, for a coefficient ratios $\phi_1 \geq \phi_2$ have been given.
As shown in the right plot of Figure~\ref{fig:OptReg2}, for a fixed distribution of $X$ (characterised by $\psi_1, \rho_1, \rho_2$) and sample size (characterised by $\gamma$) as the ratio $\phi_1/\phi_2$ increases the optimal regularisation decreases.
Following Corollary \ref{Cor:PeciseOptimalReg}, if the ratio $\phi_1/\phi_2$ is large enough, the optimal ridge regularisation parameter $\lambda$ can be $0$, corresponding to ridgeless interpolation.
We note that the negative derivative at $0$ (Corollary~\ref{Cor:PeciseOptimalReg}) and the right plot of Figure \ref{fig:OptReg2},
see also \cite{kobak2018implicit,wu2020optimal}.

\begin{figure}[!h]
    \centering
\begin{subfigure}[b]{0.45\textwidth}
        \centering
        \includegraphics[width=\textwidth]{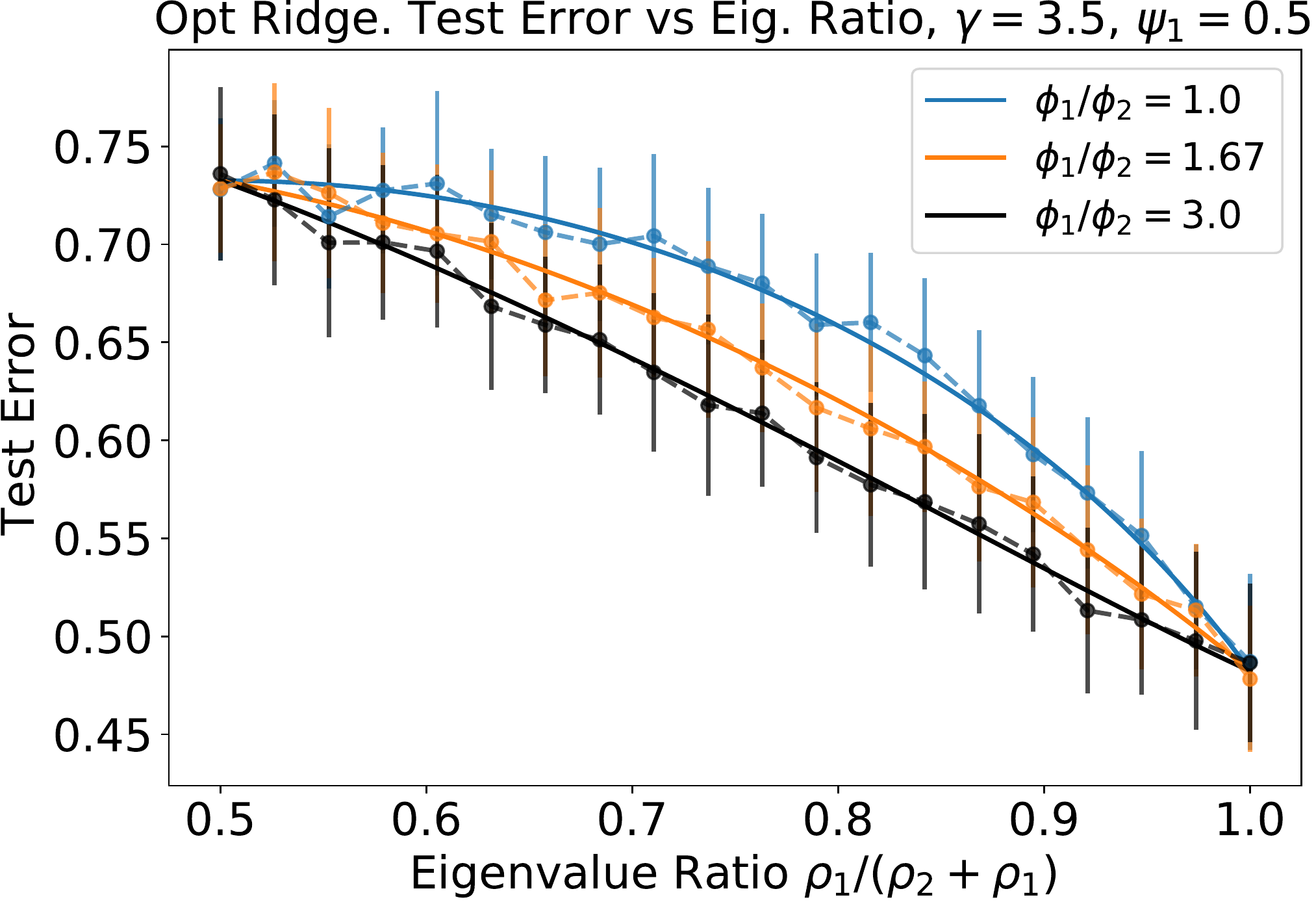}
      \end{subfigure}
    \begin{subfigure}[b]{0.45\textwidth}
      \centering
      \includegraphics[width=\textwidth]{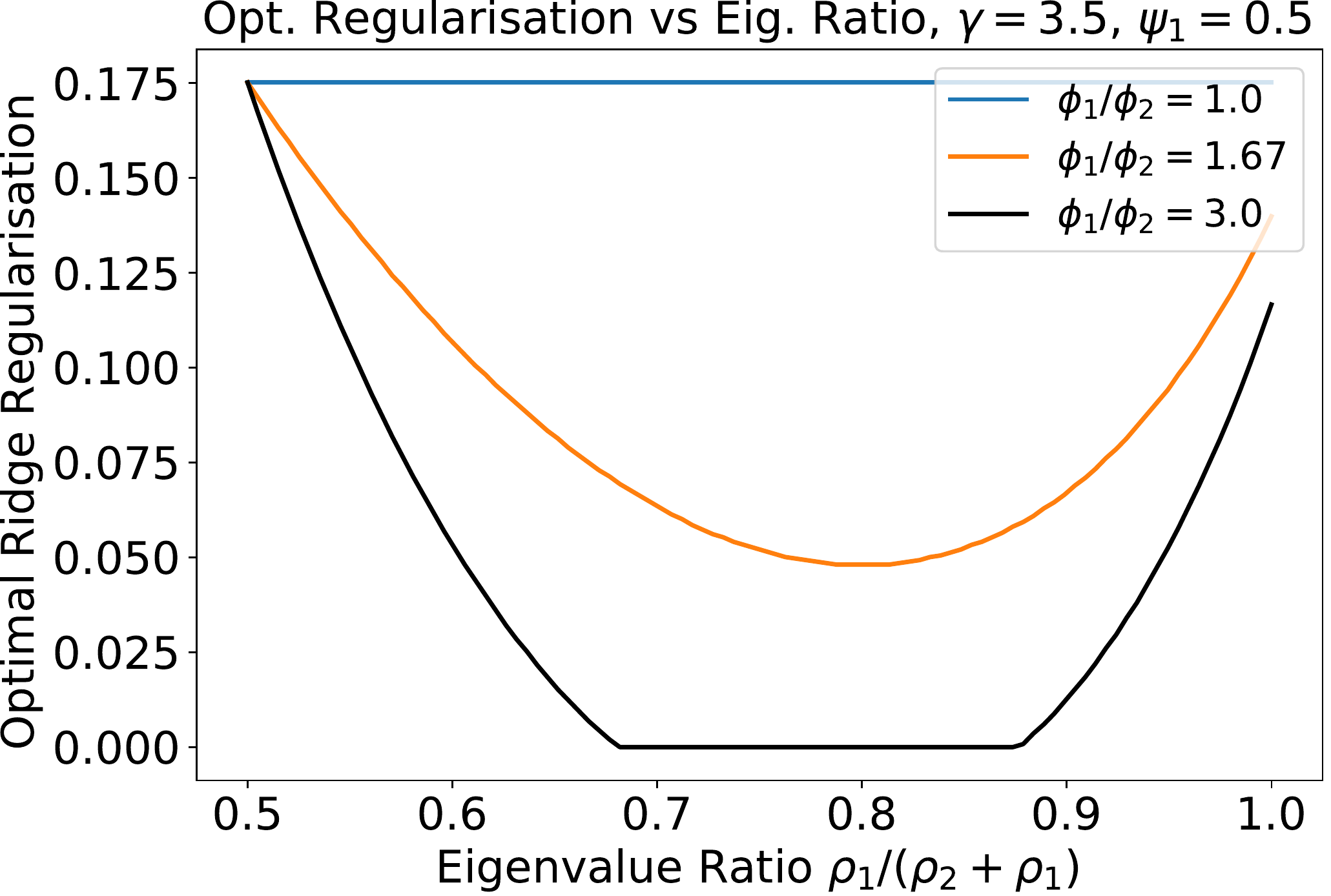}
    \end{subfigure}
    \caption{\textit{Left}: Limiting test error for optimally tuned ridge regression as described by $R_{\mathrm{Asym}}(\lambda)$, \textit{Right}: optimal regularisation computed numerically using theory. \textit{Both}: Quantities plotted against Eigenvalue ratio $\rho_1/(\rho_1+\rho_2)$. Problem parameters were $\E[\langle x,\beta^{\star}\rangle^2] = \rho_1 \phi_1\psi_1 + \rho_2 \phi_2 \psi_2 = 1 $, $\E[\|\beta^{\star}\|_2^2] = r^2(\phi_1 \psi_1 + \phi_2 \psi_2) = r^2 = 1$, $\sigma^2 = 0.05$, $\gamma = 3.5$ and $\psi_1 = 0.5$.  \textit{Left}: Dashed lines indicate simulations with $d=2^{10}$, $40$ replications, noise $\epsilon$ from standard Gaussian and covariance $\Sigma$ diagonal with $\rho_1$ on first $d_1$ co-ordinates and $\rho_2$ on remaining $d_2$.
    } 
\label{fig:OptReg2}
\end{figure}

\paragraph{Comparison with the Isotropic Model.}
In the case of a parameter $\beta^\star$ drawn from an isotropic prior $\Phi \equiv 1$ (see Section~\ref{sec:setup:data}), the optimal ridge parameter is given by $\lambda = (\sigma^2 d)/(r^2 n)$ (see Remark~\ref{rem:oracle-estimator}, as well as \cite{dobriban2018high,hastie2019surprises}).
This parameter is always positive, and is inversely proportional to the signal-to-noise ratio $r^2 / \sigma^2$.
Studying the influence of $\beta^\star$ through a general $\phi_1, \phi_2$ shows that $(1)$ optimal regularisation also depends on the coefficient decay of $\beta^\star$; $(2)$ optimal regularisation can be equal to $\lambda = 0$, which interpolates training data.
Finally, let us note that the optimal estimator of Remark~\ref{rem:oracle-estimator} (with oracle knowledge of $\Sigma, \Phi$) does \emph{not} interpolate; hence, the optimality of interpolation among the family of ridge estimators arises from a form of ``prior misspecification''.
We believe this phenomenon to extend beyond the specific case of ridge estimators.

\subsection{The Special Case of Noisy Weak Features}
\label{sec:TwoBulks:NoisyBulkReg}

In this section we consider the special case where weak features are pure noise variables, namely $\phi_2 = 0$, while their dimension is large. Such noisy weak features can be artificially introduced to the dataset, to induce an overparameterised problem.
We then refer to this technique as \emph{Noisy Feature Regularisation}, and note it corresponds to the design matrix augmentation in \cite{kobak2018implicit}. Looking to Figure \ref{fig:TestError:OverparamReg}, the ridgeless test error is then plotted against the eigenvalue ratio $\rho_2/\rho_1$  (\textit{Left}) and the number of weak features with the tuned eigenvalue ratio (\textit{Right}). 

Observe (right plot) as we increase the number of weak features (as encoded by $1/\psi_1$), and tune the eigenvalue $\rho_2$, the performance converges to optimally tuned ridge regression with the strong features only. The left plot then shows the ``regularisation path'' as a function of the eigenvalue ratio $\rho_2/\rho_1$ for some numbers of weak features $1/\psi_1$. 

\paragraph{Weak Features Can Implicitly Regularise.} The results in Sections \ref{sec:TwoBulks:ZeroRegOpt} and \ref{sec:TwoBulks:NoisyBulkReg} suggest that weak features can implicitly regularise when the ground truth is associated to a subset of stronger features. Specifically,  Section \ref{sec:TwoBulks:ZeroRegOpt} demonstrated how this can occur passively in an easy learning problem, with the weak features providing sufficient stability that zero ridge regularisation can be the optimal choice \footnote{Zero regularisation has been shown to be optimal for Random Feature regression with a high signal to noise ratio \cite{mei2019generalization} and a misspecified component. The work \cite{kobak2018implicit} numerically estimated $R^{\prime}_{\mathrm{Asym}}(\lambda)$  for a spiked covariance model and found it can be positive.
}. Meanwhile, in this section we demonstrated an active approach where weak features can purposely be added to a model and tuned similar to ridge. We note the recent work \cite{jacot2020implicit} which shows a similar implicit regularisation phenomena for kernel regression.

\begin{figure}[!h]
    \centering
\begin{subfigure}[b]{0.45\textwidth}
        \centering
        \includegraphics[width=\textwidth]{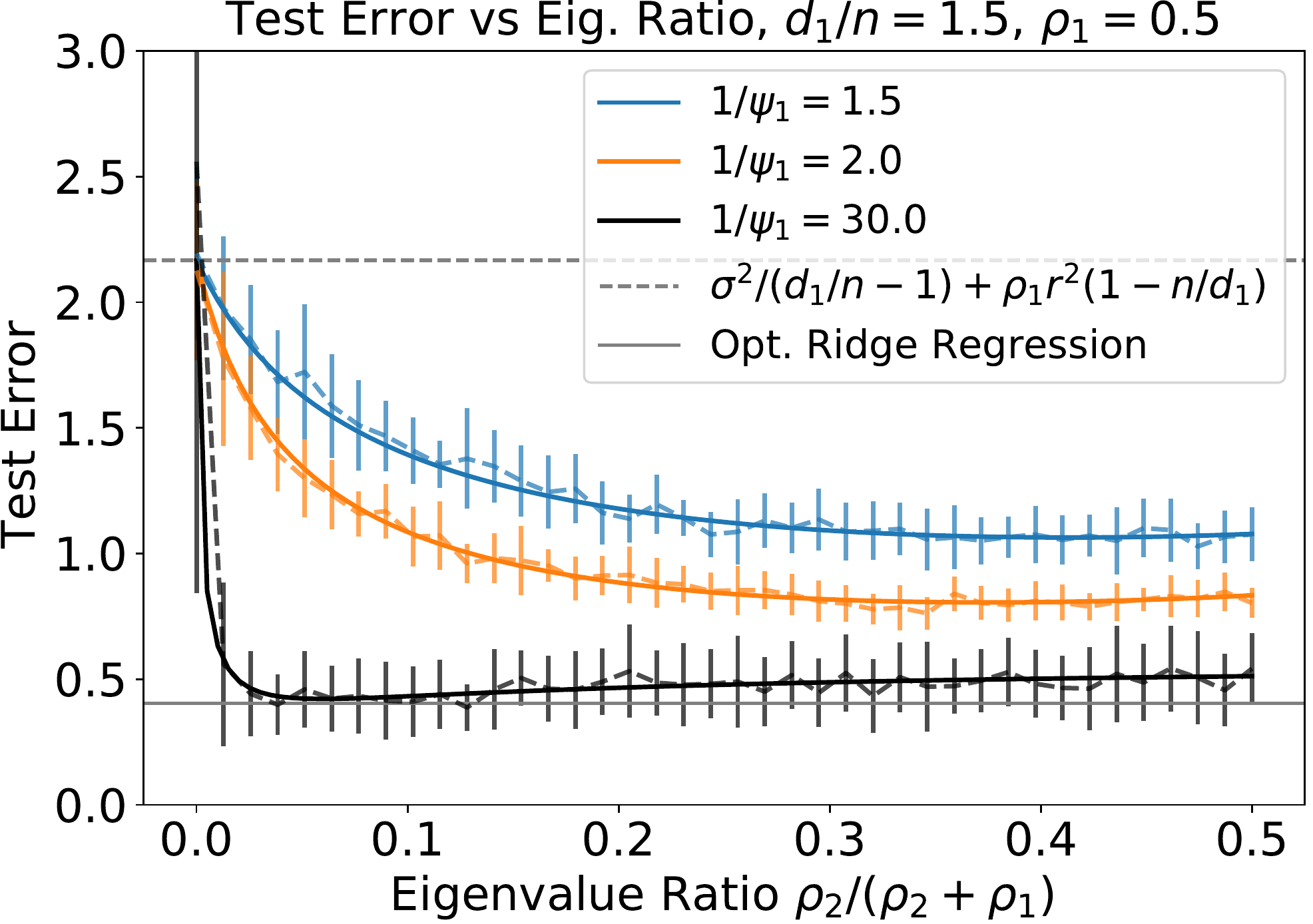}
      \end{subfigure}
    \begin{subfigure}[b]{0.45\textwidth}
        \centering
        \includegraphics[width=\textwidth]{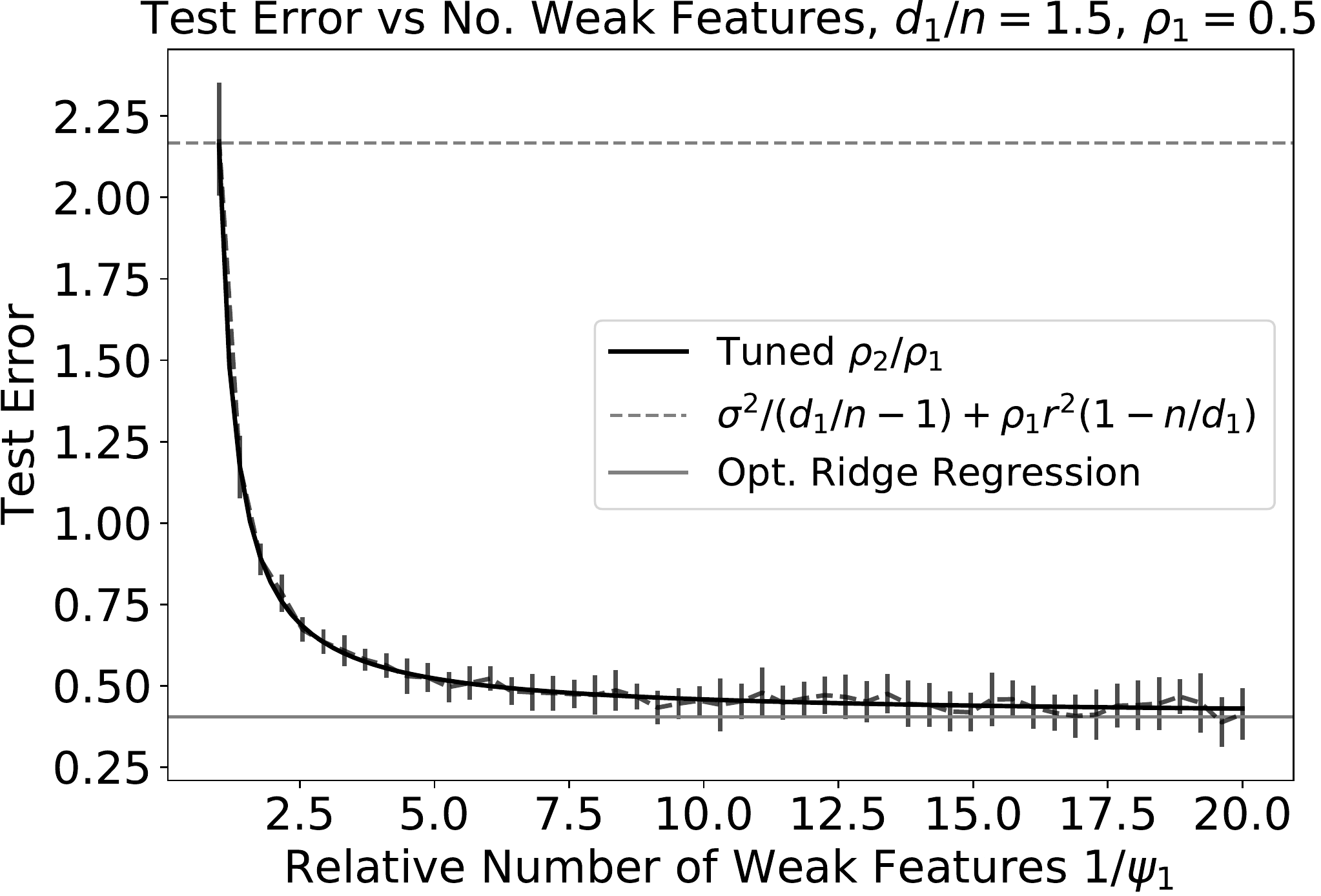}
    \end{subfigure}
    \caption{Ridgeless test error for strong and weak features model ($r^2 = \sigma^2 = 1$) against eigenvalue ratio $\rho_2/\rho_1$ (\textit{Left}) and size of noisy bulk $1/\psi_1 = d_1/d$  (\textit{Right}). Solid lines show theory computed using $v(0)$ with $\gamma\psi_1 = d_1/n = 1.5$ and $\rho_1 = 0.5$. Dashed lines are simulations with $d= 2^{8}$ (\textit{Left}) and $2^{10}$ (\textit{Right}) and $20$ replications. \textit{Solid Grey Horizontal Line}: Performance of optimally tuned ridge regression with strong features only. }
\label{fig:TestError:OverparamReg}
\end{figure}

\subsection{Ridgeless Bias and Variance}
\label{sec:TwoBulks:Ridgeless:Variance}

In this section we investigate how the ridgeless bias and variance depend on the ratio of dimension to sample size $\gamma$. 
Looking to Figure \ref{fig:TestError:Variance} the ridgeless bias and variance is plotted against the ratio of dimension to sample size in the overparameterised regime $\gamma \geq 1$ . 

Note an additional peak in the ridgeless bias and variance is observed beyond the interpolation threshold. This has only recently been empirically observed for the test error \cite{nakkiran2020optimal}, as such, these plots now theoretically verify this phenomenon. The location of the peaks naturally depends on the number of strong and weak features as well as the ambient dimension, as denoted by the vertical lines. Specifically, the peak occurs in the ridgeless bias for the ``hard'' setting when the number of samples and number of strong features are equal $n=d_1$. Meanwhile, a peak occurs in the ridgeless variance when the number of samples and strong features equal $n=d_1$, and the eigenvalue ratio is large $\rho_1 > \rho_2$. 
This demonstrates that learning curves beyond the interpolation threshold can have different characteristics due to the interplay between the covariate structure and underlying data. We conjecture this arises due to instabilities of the design matrix Moore-Penrose Pseudo-inverse, akin to the isotropic setting \cite{belkin2019two}. Since variance matches prior work \cite{dobriban2018high}, the additional peak could be previously derived. Meanwhile the peak in the bias here uses of the source condition, and thus, as far as we aware is not encompassed in prior work. 

\begin{figure}[!h]
    \centering
    \begin{subfigure}[b]{0.45\textwidth}
        \centering
        \includegraphics[width=\textwidth]{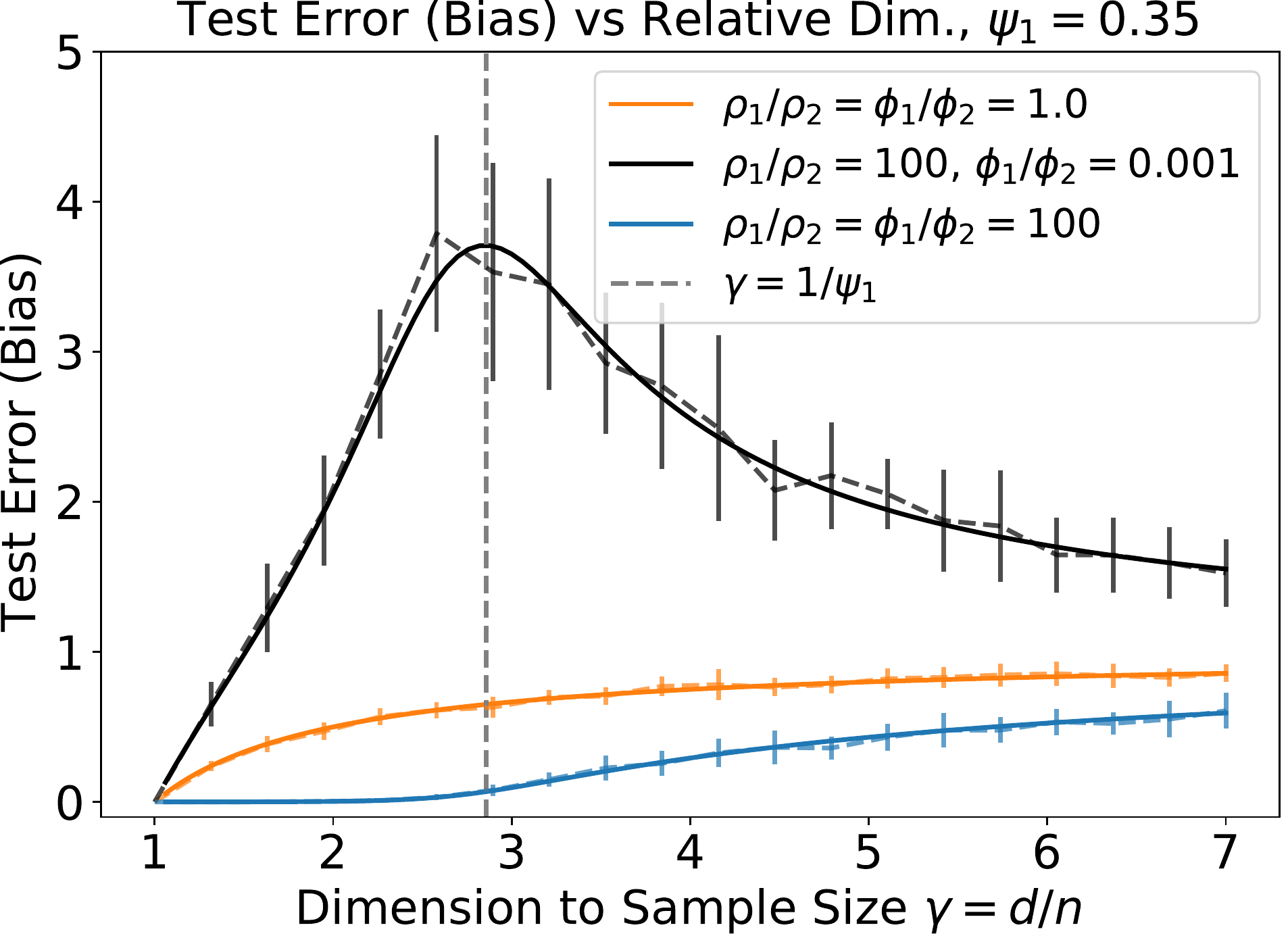}
      \end{subfigure}
    \begin{subfigure}[b]{0.45\textwidth}
        \centering
        \includegraphics[width=\textwidth]{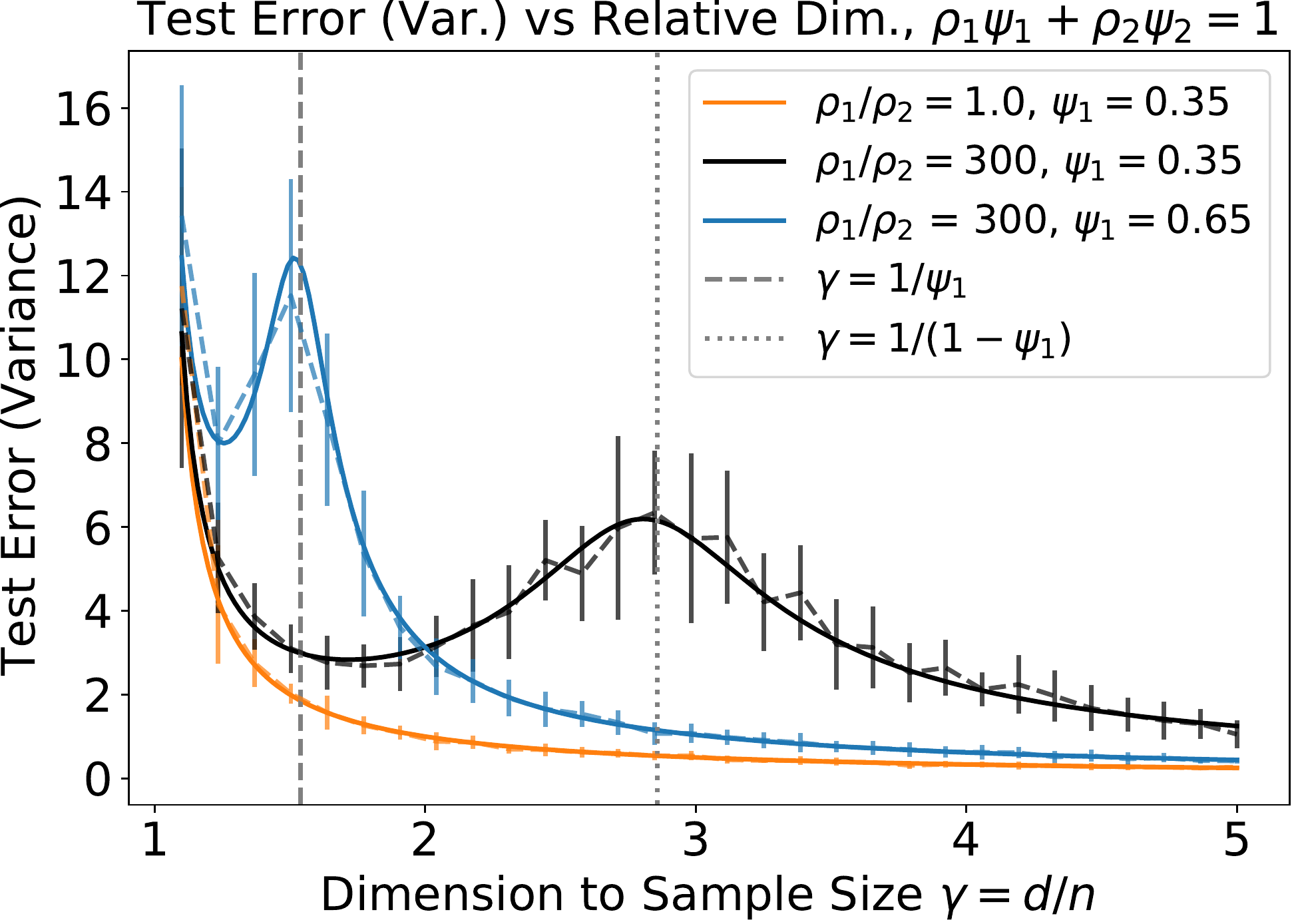}
      \end{subfigure}
    \caption{Ridgeless bias and variance for strong and weak feature model plotted against relative dimension $\gamma  =  d/n$ with various eigenvalue ratios $\rho_1/\rho_2$ and coefficients $\phi_1/\phi_2$. Solid lines are theory computed using $v(0)$ with $\psi_1  =  0.35$, $\E[\langle x,\beta^{\star}\rangle^2]  =  \rho_1 \phi_1\psi_1  +  \rho_2 \phi_2 \psi_2  =  1 $ and $\E[\|\beta^{\star}\|_2^2]  =  r^2(\phi_1 \psi_1  +  \phi_2 \psi_2)  =  1$. Dashed lines are simulations with $d  =   2^{8}$ and 20 replications. 
    }
\label{fig:TestError:Variance}
\end{figure}

\section{Conclusion}
\label{sec:conc}

In this work, we introduced a framework for studying ridge regression in a high-dimensional regime.
We characterised the limiting risk of ridge regression in terms of the dimension to sample size ratio, the spectrum of the population covariance and the coefficients of the true regression parameter along the covariance basis. This extends prior work~\cite{dicker2016ridge,dobriban2018high}, that considered an isotropic ground truth parameter. Our extension enables the  study of ``prior misspecification'', where signal strength may decrease faster or slower than postulated by the ridge estimator, and its effect on ideal regularisation.

We instantiated this general framework to a simple structure, with strong and weak features.
In this case, we show that ``ridgeless'' regression with zero regularisation can be optimal among all ridge regression estimators.
This occurs when the signal-to-noise ratio is large and when strong features (with large eigenvalue of the covariance matrix) have sufficiently more signal than weak ones. The latter condition corresponds to an ``easy'' or ``lower-dimensional'' problem, where ridge tends to over-penalise along strong features.
This phenomenon does not occur for isotropic priors, where optimal regularisation is always strictly positive in the presence of noise.
Finally, we discussed noisy weak features, which act as a form of regularisation, and concluded by showing additional peaks in ridgeless bias and variance can occur for our model.

Moving forward, it would be natural to consider non-Gaussian covariates. Given universality results in Random Matrix Theory 
we expect that the results provided here extend to the case of random vectors with independent coordinates (and linear transformations thereof).
Other structures for the ground truth and data generating process can be investigated through Theorem \ref{main:thm} by consider different functions $\Phi$ and the population eigenvalue distributions. The tradeoff between prediction and estimation error exhibited by \cite{dobriban2018high} in the isotropic case can be explored with a general source $\Phi$.

\section{Acknowledgments}
D.R. is supported by the EPSRC and MRC through the
OxWaSP CDT programme (EP/L016710/1). Part of this
work has been carried out at the Machine Learning Genoa
(MaLGa) center, Universita di Genova (IT). 
L.R. acknowledges the financial support of the European Research Council (grant SLING 819789), the AFOSR projects FA9550-17-1-0390 and BAA-AFRL-AFOSR-2016-0007 (European
Office of Aerospace Research and Development), and the
EU H2020-MSCA-RISE project NoMADS - DLV-777826. We would also like to thank the anonymous reviewers for their feedback and suggestions.

\bibliographystyle{plain}
\bibliography{References}

\begin{thebibliography}{10}

\bibitem{advani2017high}
Madhu~S Advani, Andrew~M Saxe, and Haim Sompolinsky.
\newblock High-dimensional dynamics of generalization error in neural networks.
\newblock {\em Neural Networks}, 132:428--446, 2020.

\bibitem{allen2018convergence}
Zeyuan Allen-Zhu, Yuanzhi Li, and Zhao Song.
\newblock A convergence theory for deep learning via over-parameterization.
\newblock In {\em Proceedings of the 36th International Conference on Machine
  Learning}, volume~97, pages 242--252, 2019.

\bibitem{amari2020does}
Shun-ichi Amari, Jimmy Ba, Roger Grosse, Xuechen Li, Atsushi Nitanda, Taiji
  Suzuki, Denny Wu, and Ji~Xu.
\newblock When does preconditioning help or hurt generalization?
\newblock {\em arXiv preprint arXiv:2006.10732}, 2020.

\bibitem{bai2010spectral}
Zhidong Bai and Jack~W Silverstein.
\newblock {\em Spectral analysis of large dimensional random matrices},
  volume~20.
\newblock Springer, 2010.

\bibitem{bartlett2020benign}
Peter~L. Bartlett, Philip~M. Long, G{\'a}bor Lugosi, and Alexander Tsigler.
\newblock Benign overfitting in linear regression.
\newblock {\em Proceedings of the National Academy of Sciences}, 2020.

\bibitem{bauer2007regularization}
Frank Bauer, Sergei Pereverzev, and Lorenzo Rosasco.
\newblock On regularization algorithms in learning theory.
\newblock {\em Journal of complexity}, 23(1):52--72, 2007.

\bibitem{belkin2019reconciling}
Mikhail Belkin, Daniel Hsu, Siyuan Ma, and Soumik Mandal.
\newblock Reconciling modern machine-learning practice and the classical
  bias{\textendash}variance trade-off.
\newblock {\em Proceedings of the National Academy of Sciences},
  116(32):15849--15854, 2019.

\bibitem{belkin2019two}
Mikhail Belkin, Daniel Hsu, and Ji~Xu.
\newblock Two models of double descent for weak features.
\newblock {\em SIAM Journal on Mathematics of Data Science}, 2(4):1167--1180,
  2020.

\bibitem{belkin2018overfitting}
Mikhail Belkin, Daniel~J Hsu, and Partha Mitra.
\newblock Overfitting or perfect fitting? risk bounds for classification and
  regression rules that interpolate.
\newblock In {\em Advances in neural information processing systems}, pages
  2300--2311, 2018.

\bibitem{belkin2018understand}
Mikhail Belkin, Siyuan Ma, and Soumik Mandal.
\newblock To understand deep learning we need to understand kernel learning.
\newblock In {\em International Conference on Machine Learning}, pages
  541--549. PMLR, 2018.

\bibitem{belkin2019contradict}
Mikhail Belkin, Alexander Rakhlin, and Alexandre~B. Tsybakov.
\newblock Does data interpolation contradict statistical optimality?
\newblock In {\em Proceedings of the International Conference on {Artificial
  Intelligence and Statistics (AISTATS)}}, pages 1611--1619, 2019.

\bibitem{chen2020multiple}
Lin Chen, Yifei Min, Mikhail Belkin, and Amin Karbasi.
\newblock Multiple descent: Design your own generalization curve.
\newblock {\em arXiv preprint arXiv:2008.01036}, 2020.

\bibitem{chen2011regularized}
Lin~S Chen, Debashis Paul, Ross~L Prentice, and Pei Wang.
\newblock A regularized hotelling's t 2 test for pathway analysis in proteomic
  studies.
\newblock {\em Journal of the American Statistical Association},
  106(496):1345--1360, 2011.

\bibitem{chizat2019lazy}
Lenaic Chizat, Edouard Oyallon, and Francis Bach.
\newblock On lazy training in differentiable programming.
\newblock In {\em Advances in Neural Information Processing Systems}, pages
  2933--2943, 2019.

\bibitem{dicker2016ridge}
Lee~H. Dicker.
\newblock Ridge regression and asymptotic minimax estimation over spheres of
  growing dimension.
\newblock {\em Bernoulli}, 22(1):1--37, 2016.

\bibitem{dobriban2018high}
Edgar Dobriban, Stefan Wager, et~al.
\newblock High-dimensional asymptotics of prediction: Ridge regression and
  classification.
\newblock {\em The Annals of Statistics}, 46(1):247--279, 2018.

\bibitem{du2018gradienta}
Simon Du, Jason Lee, Haochuan Li, Liwei Wang, and Xiyu Zhai.
\newblock Gradient descent finds global minima of deep neural networks.
\newblock In {\em Proceedings of the 36th International Conference on Machine
  Learning}, volume~97, pages 1675--1685, 2019.

\bibitem{du2018gradientb}
Simon~S Du, Xiyu Zhai, Barnabas Poczos, and Aarti Singh.
\newblock Gradient descent provably optimizes over-parameterized neural
  networks.
\newblock {\em International Conference on Learning Representations {(ICLR)}},
  2019.

\bibitem{elkaroui2018random}
Noureddine El~Karoui.
\newblock Random matrices and high-dimensional statistics: beyond covariance
  matrices.
\newblock In {\em Proceedings of the International Congress of Mathematicians},
  volume~4, pages 2875--2894, Rio de Janeiro, 2018.

\bibitem{engl1996regularization}
Heinz~Werner Engl, Martin Hanke, and Andreas Neubauer.
\newblock {\em Regularization of inverse problems}, volume 375.
\newblock Springer Science \& Business Media, 1996.

\bibitem{gerbelot2020asymptotic}
C{\'e}dric Gerbelot, Alia Abbara, and Florent Krzakala.
\newblock Asymptotic errors for convex penalized linear regression beyond
  gaussian matrices.
\newblock {\em arXiv preprint arXiv:2002.04372}, 2020.

\bibitem{ghorbani2019linearized}
Behrooz Ghorbani, Song Mei, Theodor Misiakiewicz, and Andrea Montanari.
\newblock Linearized two-layers neural networks in high dimension.
\newblock {\em arXiv preprint arXiv:1904.12191}, 2019.

\bibitem{hastie2019surprises}
Trevor Hastie, Andrea Montanari, Saharon Rosset, and Ryan~J Tibshirani.
\newblock Surprises in high-dimensional ridgeless least squares interpolation.
\newblock {\em arXiv preprint arXiv:1903.08560}, 2019.

\bibitem{hoerl1970ridge}
Arthur~E Hoerl and Robert~W Kennard.
\newblock Ridge regression: Biased estimation for nonorthogonal problems.
\newblock {\em Technometrics}, 12(1):55--67, 1970.

\bibitem{jacot2018neural}
Arthur Jacot, Franck Gabriel, and Cl{\'e}ment Hongler.
\newblock Neural tangent kernel: Convergence and generalization in neural
  networks.
\newblock In {\em Advances in neural information processing systems}, pages
  8571--8580, 2018.

\bibitem{jacot2020implicit}
Arthur Jacot, Berfin Simsek, Francesco Spadaro, Cl{\'e}ment Hongler, and Franck
  Gabriel.
\newblock Implicit regularization of random feature models.
\newblock In {\em International Conference on Machine Learning}, pages
  4631--4640. PMLR, 2020.

\bibitem{johnstone2001distribution}
Iain~M Johnstone.
\newblock On the distribution of the largest eigenvalue in principal components
  analysis.
\newblock {\em Annals of statistics}, pages 295--327, 2001.

\bibitem{jolliffe1982note}
Ian~T Jolliffe.
\newblock A note on the use of principal components in regression.
\newblock {\em Journal of the Royal Statistical Society: Series C (Applied
  Statistics)}, 31(3):300--303, 1982.

\bibitem{karoui2013asymptotic}
Noureddine~El Karoui.
\newblock Asymptotic behavior of unregularized and ridge-regularized
  high-dimensional robust regression estimators: rigorous results.
\newblock {\em arXiv preprint arXiv:1311.2445}, 2013.

\bibitem{karoui2011geometric}
Noureddine~El Karoui and Holger K{\"o}sters.
\newblock Geometric sensitivity of random matrix results: consequences for
  shrinkage estimators of covariance and related statistical methods.
\newblock {\em arXiv preprint arXiv:1105.1404}, 2011.

\bibitem{kobak2018implicit}
Dmitry Kobak, Jonathan Lomond, and Benoit Sanchez.
\newblock The optimal ridge penalty for real-world high-dimensional data can be
  zero or negative due to the implicit ridge regularization.
\newblock {\em Journal of Machine Learning Research}, 21(169):1--16, 2020.

\bibitem{lecun1989backpropagation}
Yann LeCun, Bernhard Boser, John~S Denker, Donnie Henderson, Richard~E Howard,
  Wayne Hubbard, and Lawrence~D Jackel.
\newblock Backpropagation applied to handwritten zip code recognition.
\newblock {\em Neural computation}, 1(4):541--551, 1989.

\bibitem{lecun1998gradient}
Yann LeCun, L{\'e}on Bottou, Yoshua Bengio, and Patrick Haffner.
\newblock Gradient-based learning applied to document recognition.
\newblock {\em Proceedings of the IEEE}, 86(11):2278--2324, 1998.

\bibitem{ledoit2011eigenvectors}
Olivier Ledoit and Sandrine P{\'e}ch{\'e}.
\newblock Eigenvectors of some large sample covariance matrix ensembles.
\newblock {\em Probability Theory and Related Fields}, 151(1-2):233--264, 2011.

\bibitem{liang2018just}
Tengyuan Liang, Alexander Rakhlin, et~al.
\newblock Just interpolate: Kernel “ridgeless” regression can generalize.
\newblock {\em Annals of Statistics}, 48(3):1329--1347, 2020.

\bibitem{liao2019large}
Zhenyu Liao and Romain Couillet.
\newblock A large dimensional analysis of least squares support vector
  machines.
\newblock {\em IEEE Transactions on Signal Processing}, 67(4):1065--1074, 2019.

\bibitem{liao2019inner}
Zhenyu Liao and Romain Couillet.
\newblock On inner-product kernels of high dimensional data.
\newblock In {\em 2019 IEEE 8th International Workshop on Computational
  Advances in Multi-Sensor Adaptive Processing (CAMSAP)}, pages 579--583. IEEE,
  2019.

\bibitem{louart2018random}
Cosme Louart, Zhenyu Liao, Romain Couillet, et~al.
\newblock A random matrix approach to neural networks.
\newblock {\em The Annals of Applied Probability}, 28(2):1190--1248, 2018.

\bibitem{mahdaviyeh2019asymptotic}
Yasaman Mahdaviyeh and Zacharie Naulet.
\newblock Asymptotic risk of least squares minimum norm estimator under the
  spike covariance model.
\newblock {\em arXiv preprint arXiv:1912.13421}, 2019.

\bibitem{marvcenko1967distribution}
Vladimir~A Mar{\v{c}}enko and Leonid~Andreevich Pastur.
\newblock Distribution of eigenvalues for some sets of random matrices.
\newblock {\em Mathematics of the USSR-Sbornik}, 1(4):457, 1967.

\bibitem{mathe2003geometry}
Peter Math{\'e} and Sergei~V Pereverzev.
\newblock Geometry of linear ill-posed problems in variable hilbert scales.
\newblock {\em Inverse problems}, 19(3):789, 2003.

\bibitem{mei2019generalization}
Song Mei and Andrea Montanari.
\newblock The generalization error of random features regression: Precise
  asymptotics and double descent curve.
\newblock {\em arXiv preprint arXiv:1908.05355}, 2019.

\bibitem{mitra2019understanding}
Partha~P Mitra.
\newblock Understanding overfitting peaks in generalization error: Analytical
  risk curves for $ l\_2 $ and $ l\_1 $ penalized interpolation.
\newblock {\em arXiv preprint arXiv:1906.03667}, 2019.

\bibitem{muthukumar2020harmless}
Vidya Muthukumar, Kailas Vodrahalli, Vignesh Subramanian, and Anant Sahai.
\newblock Harmless interpolation of noisy data in regression.
\newblock {\em IEEE Journal on Selected Areas in Information Theory}, 2020.

\bibitem{nakkiran2020optimal}
Preetum Nakkiran, Prayaag Venkat, Sham Kakade, and Tengyu Ma.
\newblock Optimal regularization can mitigate double descent.
\newblock {\em arXiv preprint arXiv:2003.01897}, 2020.

\bibitem{paul2007asymptotics}
Debashis Paul.
\newblock Asymptotics of sample eigenstructure for a large dimensional spiked
  covariance model.
\newblock {\em Statistica Sinica}, pages 1617--1642, 2007.

\bibitem{silverstein1995analysis}
Jack~W Silverstein and Sang-Il Choi.
\newblock Analysis of the limiting spectral distribution of large dimensional
  random matrices.
\newblock {\em Journal of Multivariate Analysis}, 54(2):295--309, 1995.

\bibitem{silverstein1992signal}
Jack~W Silverstein and Patrick~L Combettes.
\newblock Signal detection via spectral theory of large dimensional random
  matrices.
\newblock {\em IEEE Transactions on Signal Processing}, 40(8):2100--2105, 1992.

\bibitem{spigler2019jamming}
S~Spigler, M~Geiger, S~d'Ascoli, L~Sagun, G~Biroli, and M~Wyart.
\newblock A jamming transition from under-to over-parametrization affects
  generalization in deep learning.
\newblock {\em Journal of Physics A: Mathematical and Theoretical},
  52(47):474001, 2019.

\bibitem{steinwart2009optimal}
Ingo Steinwart, Don Hush, and Clint Scovel.
\newblock Optimal rates for regularized least squares regression.
\newblock In {\em Proceedings of the 22nd Annual Conference on Learning Theory
  {(COLT)}}, pages 79--93, 2009.

\bibitem{tikhonov1963regularization}
Andrey~N. Tikhonov.
\newblock Solution of incorrectly formulated problems and the regularization
  method.
\newblock {\em Soviet Mathematics Doklady}, 4:1035--1038, 1963.

\bibitem{tulino2004random}
Antonia~M Tulino, Sergio Verd{\'u}, et~al.
\newblock Random matrix theory and wireless communications.
\newblock {\em Foundations and Trends{\textregistered} in Communications and
  Information Theory}, 1(1):1--182, 2004.

\bibitem{wu2020optimal}
Denny Wu and Ji~Xu.
\newblock On the optimal weighted l\_2 regularization in overparameterized
  linear regression.
\newblock {\em arXiv preprint arXiv:2006.05800}, 2020.

\bibitem{xu2019number}
Ji~Xu and Daniel~J Hsu.
\newblock On the number of variables to use in principal component regression.
\newblock In {\em Advances in Neural Information Processing Systems}, pages
  5094--5103, 2019.

\bibitem{yao2015sample}
Jianfeng Yao, Shurong Zheng, and ZD~Bai.
\newblock {\em Sample covariance matrices and high-dimensional data analysis}.
\newblock Cambridge University Press Cambridge, 2015.

\bibitem{zhang2016understanding}
Chiyuan Zhang, Samy Bengio, Moritz Hardt, Benjamin Recht, and Oriol Vinyals.
\newblock Understanding deep learning requires rethinking generalization.
\newblock {\em arXiv preprint arXiv:1611.03530}, 2016.

\end{thebibliography}

\newpage
\onecolumn

\appendix

\setcounter{tocdepth}{2}
\tableofcontents

\section{Proofs for Ridge Regression}
In this section we provide the calculations associated to ridge regression. 
Section \ref{sec:ProofProp} provides the proof of Proposition \ref{prop:reduction-source-condition}.
Section \ref{sec:ProofOracle} provides the calculation for the oracle estimator presented in remark \ref{rem:oracle-estimator}. 
Section~\ref{sec:ProofRR:Prelims} provides some preliminary calculations related to random matrix theory.
Section \ref{sec:ProofRR:Thm} gives the proof of Theorem \ref{main:thm}.
Section \ref{sec:proofRR:Corr} provides the proof of Corollary \ref{Cor:SpikedCovar}. 
Section \ref{sec:ProofRR:TwoBulks} provides the calculations associated to the strong and weak features model. 

\subsection{Proof of Proposition \ref{prop:reduction-source-condition}}
\label{sec:ProofProp}
In the proof of this result, it is useful to indicate dependence on the true parameter $\beta^\star$ by denoting the risk $R_{\beta^\star}(\cdot) = \E_{\epsilon}[\|\Sigma^{1/2}(\cdot - \beta^{\star})\|_2^2]  + \sigma^2$.
We also denote by $\excessrisk_{\beta^\star} (\beta) = R_{\beta^\star} (\beta) - R_{\beta^\star} (\beta^\star) = \| \Sigma^{1/2} (\beta - \beta^\star) \|^2_2$ the excess risk of $\beta \in \R^d$ when the true parameter is $\beta^\star \in \R^d$.

\begin{lemma}
  \label{lem:rotation-invariance}
  Let $V_1, \dots, V_k \subset \R^d$ (with $k \leq d$) denote the eigenspaces of $\Sigma$.
  For $j = 1, \dots, k$, let $U_j \in O (V_j)$ be a linear isometry of $V_j$, and let $U \in \R^{d \times d}$ be the linear isometry defined by $U v = U_j v$ for $v \in V_j$, $j= 1, \dots, k$.
  Then,
  \begin{equation*}
    \E_{X,\eps} [ R_{\beta} (\wh \beta_\lambda) - R_{\beta} (\beta) ]
    = \E_{X,\eps} [ R_{U \beta} (\wh \beta_\lambda) - R_{U \beta} (U \beta) ]
    \, .
  \end{equation*}
\end{lemma}

\begin{proof}[Proof of Lemma~\ref{lem:rotation-invariance}]
  Denote $\beta' = U \beta$, as well as
  $X' = X U^{-1}$ and $x' = U^{-1} x$.
  Let $\wh \beta_\lambda'$ the Ridge estimator computed on data $(X',Y)$, namely
  \begin{equation*}
    \wh \beta_\lambda'
    = (X'^{\top} X' + \lambda n I)^{-1} X'^\top Y
    = ( U X^\top X U^{-1} + \lambda n I )^{-1} U X^\top Y
    = U \wh \beta_\lambda
    \, .
  \end{equation*}
  Then, $y = \langle \beta, x\rangle + \sigma \eps = \langle \beta', x'\rangle + \sigma \eps$, hence the best linear predictor of $y$ based on $x'$ is $\beta'$.
  In addition, $x'$ has distribution $\gaussdist (0, U^{-1} \Sigma U) = \gaussdist (0, \Sigma)$, where $U^{-1} \Sigma U = \Sigma$ comes from the fact that $U$ is an isometry on the eigenspaces $V_j$ of $\Sigma$.
  This implies that
  $(X',\eps)$ has the same distribution as $(X,\eps)$, and thus
  $\E_{\eps, X} [ \excessrisk_{\beta'} (\wh \beta'_\lambda) ] = \E_{\eps, X} [ \excessrisk_{\beta'} (\wh \beta_\lambda) ]$.
  On the other hand,
  \begin{equation*}
    \excessrisk_{\beta'} (\wh \beta_\lambda') = 
  \| \Sigma^{1/2} (\wh \beta_\lambda' - \beta') \|^2_2 = \| \Sigma^{1/2} U (\wh \beta_\lambda - \beta) \|^2_2 = \| \Sigma^{1/2} (\wh \beta_\lambda - \beta) \|^2_2 = \excessrisk_{\beta} (\wh \beta_\lambda)
  \end{equation*}
  (note that $\| \Sigma^{1/2} U \cdot \|^2_2 = \| U \Sigma^{1/2} \cdot \|^2_2 = \| \Sigma^{1/2} \cdot \|^2_2$ as $U$ commutes with $\Sigma^{1/2}$ and is an isometry),
  so that $\E_{\eps, X} [ \excessrisk_{\beta'} (\wh \beta'_\lambda) ] = \E_{\eps, X} [ \excessrisk_{\beta} (\wh \beta_\lambda) ]$.
  This proves that
  $\E_{\eps, X} [ \excessrisk_{\beta'} (\wh \beta_\lambda) ] = \E_{\eps, X} [ \excessrisk_{\beta} (\wh \beta_\lambda) ]$.
\end{proof}

We now turn to the proof of Proposition~\ref{prop:reduction-source-condition}:

\begin{proof}[Proof of Proposition~\ref{prop:reduction-source-condition}]
    Let $V_1, \dots, V_k$ denote the eigenspaces of $\Sigma$, with distinct eigenvalues $\tau_1' > \dots > \tau_k'$.
  Let $U_1, \dots, U_k$ be independent random isometries, where $U_j$ is distributed according to the uniform (Haar) measure on the orthogonal group of $V_j$.
  Define $U$ to be the random isometry acting as $U_j$ on $V_j$, and let $\beta = U \beta^\star$ and $\Pi$ its distribution.

  Note that $U$ is of the form of Lemma~\ref{lem:rotation-invariance}, hence $\E_{X,\eps} [\excessrisk_{U \beta^\star} (\wh \beta_\lambda)] = \E_{X,\eps} [\excessrisk_{\beta^\star} (\wh \beta_\lambda)]$ and thus
  \begin{equation}
    \label{eq:proof-reduction-source-inv}
    \E_{\beta \sim \Pi} \E_{X,\eps} [\excessrisk_{\beta} (\wh \beta_\lambda)]
    = \E_U \E_{X,\eps} [\excessrisk_{U \beta^\star} (\wh \beta_\lambda)]
    = \E_{X,\eps} [\excessrisk_{\beta^\star} (\wh \beta_\lambda)]
    \, .
  \end{equation}

  Now, let $\beta_j' \in V_j$ be the orthogonal projection of $\beta^\star$ on $V_j$, so that $U \beta^\star = \sum_{j=1}^k U_j \beta_j'$.  
  We have $\E [ U \beta^\star ] = 0$ since $\E [U_j] = 0$ for all $j$.
  In addition, the distribution of $U_j \beta^\star$ is invariant by rotation (since $R_j U_j$ has the same distribution as $U_j$ for any fixed rotation $R_j$), hence $\E [ (U_j \beta_j') (U_j \beta_j')^\top ] = t_j I_{V_j}$ (with $I_{V_j}$ the identity on $V_j$), where letting $d_j = \dim (V_j)$,
  \begin{equation}
    \label{eq:proof-red-var}
    d_j \cdot t_j = \tr \E [ (U_j \beta_j') (U_j \beta_j')^\top ]
    = \E [ \| U_j \beta_j' \|^2_2  ]
    = \| \beta'_j \|^2_2 \, ,
  \end{equation}
  hence $t_j = \| \beta'_j \|^2_2 / d_j$.
  In addition, if $j \neq l$, by independence of $U_j,U_l$,
  \begin{equation}
    \label{eq:proof-red-cov}
    \E [ (U_j \beta_j') (U_l \beta_l')^\top ]
    = \E [ U_j \beta_j' \beta_l'^\top U_l^\top ]
    = \E [ U_j ] \beta_j' \beta_l'^\top \E [ U_l ]^\top
    = 0
    \, .
  \end{equation}
  Hence, $\Pi$ has covariance $\sum_{j=1}^k (\| \beta_j' \|^2 / d_j) I_{V_j}$, which is precisely $\Phi (\Sigma) / d$ where $\Phi$ is defined as in~\eqref{eq:equivalent-source}.
  The proof is concluded by noting that the quantity
  $\E_{\eps,X}\excessrisk_{\beta} (\wh \beta_\lambda)$
  is quadratic in $\beta$, hence if $\Pi'$ is another distribution on $\R^d$ with mean $0$ and covariance $\Phi (\Sigma) / d$, then
  $\E_{\beta \sim \Pi'} \E_{\eps,X}\excessrisk_{\beta} (\wh \beta_\lambda) = \E_{\beta \sim \Pi} \E_{\eps,X}\excessrisk_{\beta} (\wh \beta_\lambda) = \E_{\eps,X}\excessrisk_{\beta^\star} (\wh \beta_\lambda)$.
\end{proof}

\subsection{Proof of Oracle Estimator (Remark \ref{rem:oracle-estimator})}
\label{sec:ProofOracle}

  Since the risk is quadratic, the average risk (integrated over the prior) of any estimator linear in $Y$ only depends on the first two moments of the prior, hence one can assume that the prior is Gaussian (namely, $\gaussdist (0, r^2 \Phi (\Sigma) / d)$) without loss of generality.
  In this case, a standard computation shows that the posterior is $\gaussdist (\wt \beta, [ {X^\top X} + (\sigma^2d/r^2) \Phi(\Sigma)^{-1} ]^{-1})$, where $\wt \beta$ is the estimator defined in~\eqref{eq:oracle-estimator}.
  Finally, since the risk is quadratic, the Bayes-optimal estimator is the posterior mean, which corresponds to $\wt \beta$.

\subsection{Random Matrix Theory Preliminaries}
\label{sec:ProofRR:Prelims}
We now introduce some useful properties of the Stieltjes transform as well as its companion transform. Firstly, we know the companion transform satisfies the Silverstein equation \cite{silverstein1992signal,silverstein1995analysis} 
\begin{align}
\label{equ:CompanionTransform:StationaryPoint}
    - \frac{1}{v(z)}
    = 
    z - \gamma \int \frac{\tau}{1 + \tau v(z)} dH(\tau).
\end{align}
We then have for $z \in \mathcal{S} := \{u  + i v : v \not= 0, \text{ or } v = 0, u > 0\}$, the companion transform $v(z)$ is the unique solution to the Silverstein equation with $v(z) \in \mathcal{S}$ such that the sign of the imaginary part is preserved $\mathrm{sign}(\mathrm{Im}(v(z))) - \mathrm{sign}(\mathrm{Im}(z))$. The above can then be differentiated with respect to $z$ to obtain a formula for $v^{\prime}(z)$ in terms of $v(z)$:
\begin{align*}
    \frac{\partial v(z)}{\partial z}
    = 
    \Big( \frac{1}{v(z)^2} - \gamma \int \frac{\tau^2}{(1+\tau v(z))^2} dH(\tau) \Big)^{-1}
\end{align*}
Meanwhile from from the equality $\gamma (m(z) + 1/z) = v(z) + 1/z$ we note that we have the following equalities 
\begin{align}
\label{equ:STransformToComp}
    1-\gamma(1-\lambda m(-\lambda)) & = \lambda v(-\lambda)\\
    1-\lambda m(-\lambda) & = \gamma^{-1}(1-\lambda v(-\lambda)) \nonumber \\
    m(-\lambda) - \lambda m^{\prime}(-\lambda) & = \gamma^{-1}(v(-\lambda) - \lambda v^{\prime}(-\lambda))
    \nonumber
\end{align}
which we will readily use to simplify/rewrite a number of the limiting functions.

\subsection{Proof of Theorem \ref{main:thm}}
\label{sec:ProofRR:Thm}
We begin with the decomposition into bias and variance terms following \cite{dobriban2018high}.
The difference for the ridge parameter can be denoted
\begin{align*}
    \wh \beta_{\lambda} 
    - \beta^{\star}
    =  - \lambda \big( \frac{X^{\top}X}{n} + \lambda I \big)^{-1} \beta^{\star} 
    + 
    \sigma \big( \frac{X^{\top} X}{n} + \lambda I \big)^{-1} \frac{X^{\top} \epsilon}{n}
\end{align*}
And thus taking expectation with respect to the noise in the observations $\epsilon$ 
\begin{align*}
     \E_{\epsilon}[ R(\wh \beta_{\lambda})]
    - 
    R(\beta^{\star}) 
     & = 
    \E_{\epsilon}[\| \Sigma^{1/2}(\wh \beta_{\lambda} - \beta^{\star}) \|_2^2] \\
    & = 
    \E_{\epsilon}[\| \Sigma^{1/2}(\wh \beta_{\lambda} - \E_{\epsilon}[\wh \beta_{\lambda}] ) \|_2^2]
    + 
    \|\Sigma^{1/2}(\E_{\epsilon}[\wh \beta_{\lambda}] - \beta^{\star})\|_2^2\\
    & =
    \sigma^2 
    \E_{\epsilon}\big[ 
    \big\| \Sigma^{1/2} \big( \frac{X^{\top} X}{n} + \lambda I \big)^{-1} \frac{X^{\top} \epsilon}{n}
    \big\|_2^2
    \big]
    +
    \lambda^2 \| \Sigma^{1/2} \big( \frac{X^{\top}X}{n} - \lambda I \big)^{-1} \beta^{\star}\|_2^2\\
    & = \frac{\sigma^2}{n}
    \trace\Big( \Big( \frac{X^{\top}X}{n} + \lambda I \Big)^{-1} 
    \Sigma \Big( \frac{X^{\top}X}{n} + \lambda I \Big)^{-1} 
    \frac{X^{\top} X}{n}
    \Big) \\
    & \quad\quad 
    \lambda^2 
    \trace \Big( 
    (\beta^{\star})^{\top} 
    \Big( \frac{X^{\top}X}{n} + \lambda I \Big)^{-1} \Sigma
    \Big( \frac{X^{\top}X}{n} + \lambda I \Big)^{-1}
    \beta^{\star}
    \Big)
\end{align*}
Taking expectation with respect to $\E_{\beta^{\star}}$ we arrive at 
\begin{align*}
    & \E_{\beta^{\star}}[
    \E_{\epsilon}[ R(\wh \beta_{\lambda})]
    - 
    R(\beta^{\star})
    ]
    = 
    \frac{\sigma^2}{n}
    \trace\Big( \Big( \frac{X^{\top}X}{n} + \lambda I \Big)^{-1} 
    \Sigma \Big( \frac{X^{\top}X}{n} + \lambda I \Big)^{-1} 
    \frac{X^{\top} X}{n}
    \Big) \\
    & \quad\quad 
    \frac{\lambda^2 r^2}{d}
    \trace \Big( 
    \Big( \frac{X^{\top}X}{n} + \lambda I \Big)^{-1} \Sigma
    \Big( \frac{X^{\top}X}{n} + \lambda I \Big)^{-1}
    \Phi(\Sigma)
    \Big) \\
    & = 
    \sigma^2 \gamma \frac{1}{d}
    \trace\Big( \Big( \frac{X^{\top}X}{n} + \lambda I \Big)^{-1} 
    \Sigma 
    \Big) 
    -\lambda 
    \sigma^2\gamma \frac{1 }{d}
    \trace\Big( \Big( \frac{X^{\top}X}{n} + \lambda I \Big)^{-2} 
    \Sigma
    \Big)\\
    & \quad\quad +
    \frac{\lambda^2 r^2}{d}
    \trace \Big( 
    \Big( \frac{X^{\top}X}{n} + \lambda I \Big)^{-1} \Sigma
    \Big( \frac{X^{\top}X}{n} + \lambda I \Big)^{-1}
    \Phi(\Sigma)
    \Big) \\
\end{align*}
It is now a matter of showing the asymptotic almost sure convergence of the following three functionals
\begin{align*}
    & \frac{1}{d}
    \trace\Big( \Big( \frac{X^{\top}X}{n} + \lambda I \Big)^{-1} 
    \Sigma 
    \Big), 
    \quad 
    \frac{1}{d}
    \trace\Big( \Big( \frac{X^{\top}X}{n} + \lambda I \Big)^{-2} 
    \Sigma
    \Big)\\
    & \quad 
    \text{ and } 
    \frac{1}{d} 
    \trace \Big( 
    \Big( \frac{X^{\top}X}{n} + \lambda I \Big)^{-1} \Sigma
    \Big( \frac{X^{\top}X}{n} + \lambda I \Big)^{-1}
    \Phi(\Sigma)
    \Big)
\end{align*}
The limit of the first trace quantity comes directly from \cite{ledoit2011eigenvectors} meanwhile the limit of the second trace quantity is proven in \cite{dobriban2018high}. The third trace quantity depends upon the source condition $\Phi$ and computing its limit is one of the main technical contributions of this work. The limits for these objects is summarised within the following Lemma, the proof of which provides the key steps for computing the limit involving the source function. 
\begin{lemma}
\label{lem:Limits}
Under the assumptions of Theorem \ref{main:thm} for any $\lambda > 0$ we have almost surely as $n,d \rightarrow \infty$ with $d/n = \gamma$ 
\begin{align}
\label{equ:lem:limits:1}
    & \frac{1}{d}
    \trace\Big( \Big( \frac{X^{\top}X}{n} + \lambda I \Big)^{-1} 
    \Sigma 
    \Big)
    \rightarrow 
    \frac{1 - \lambda m(-\lambda )}{1 - \gamma (1 -\lambda m(-\lambda)) } \\
\label{equ:lem:limits:2}
    & \frac{1}{d}
    \trace\Big( \Big( \frac{X^{\top}X}{n} + \lambda I \Big)^{-2} 
    \Sigma
    \Big)
    \rightarrow 
    \frac{ m(-\lambda) - \lambda m^{\prime}(-\lambda) }{\big( 1 - \gamma(1 - \lambda m(-\lambda))\big)^2 }\\
\label{equ:lem:limits:3}
    & \frac{1}{d} 
    \trace \Big( 
    \Big( \frac{X^{\top}X}{n} + \lambda I \Big)^{-1} \Sigma
    \Big( \frac{X^{\top}X}{n} + \lambda I \Big)^{-1}
    \Phi(\Sigma)
    \Big)
    \rightarrow 
    \frac{ \Theta^{\Phi}(-\lambda) + \lambda \frac{ \partial \Theta^{\Phi}(-\lambda)}{\partial \lambda } }{\big( 1 - \gamma ( 1 - \lambda m(-\lambda)) \big)^{2}}
\end{align}
\end{lemma}
The result is arrived at by plugging in the above limits and noting from the definition of the Companion Transform $v$ that $1-\gamma(1-\lambda m(-\lambda)) = \lambda v(-\lambda)$, $1-\lambda m(-\lambda) = \gamma^{-1}(1-\lambda v(-\lambda))$ and, taking derivatives, $m(-\lambda) - \lambda m^{\prime}(-\lambda) = \gamma^{-1}(v(-\lambda) - \lambda v^{\prime}(-\lambda))$.
The proof of Lemma~\ref{lem:Limits}, which is the key technical step in the proof of Theorem~\ref{main:thm}, is provided in Appendix~\ref{app:proof-lemma-rmt}.

\subsection{Proof of Corollary \ref{Cor:SpikedCovar}}
\label{sec:proofRR:Corr}

In this section we provide the proof of Corollary \ref{Cor:SpikedCovar}. It will be broken into three parts associated to the three cases $\Phi(x) =x$, $\Phi(x) = 1$ and $\Phi(x) = 1/x$. 

\subsubsection{Case: $\Phi(x) = x$}
The purpose of this section is to demonstrate, in the case $\Phi(x) = x$, how the functional 
$\Theta^{\Phi}(-\lambda) + \lambda \frac{ \partial \Theta^{\Phi}(-\lambda)}{\partial \lambda } $ can be written in terms of the Stieltjes Transform $m(z)$. For this particular choice of $\Phi$ the asymptotics were calculated in \cite{chen2011regularized}, see also Lemma 7.9 in \cite{dobriban2018high}. 
We therefore repeat this calculation for completeness.
Now, in this case we have 
\begin{align*}
    \Theta^{\Phi}(z) = \int \frac{\tau}{\tau(1 - \gamma(1+ z m(-\lambda)))  - z} dH(\tau)
\end{align*}
Following the steps are the start of the proof for Lemma 2.2 in \cite{ledoit2011eigenvectors}, consider $1+ z m(z)$ 
\begin{align*}
    1 + z m(z) &= \int 1 + \frac{z}{\tau(1-\gamma(1+z m(z))) - z} d H(\tau) \\
	& = \int \frac{ \tau(1-\gamma(1+z m(z)))}{ \tau(1-\gamma(1+z m(z))) - z}d H(\tau) \\
	& = (1-\gamma(1+z m(z))) \Theta^{\Phi}(z)
\end{align*}
Solving for $\Theta^{\Phi}(z)$ we have
\begin{align*}
    \Theta^{\Phi}(z)  = \frac{1 + z m(z)}{1-\gamma(1+z m(z))} = \frac{1}{\gamma}\big( \frac{1}{ 1-\gamma(1+z m(z)) } - 1 \big)
\end{align*}
Picking $z = - \lambda$ and differentiating with respect to $\lambda$ we get 
\begin{align*}
    \frac{\partial \Theta^{\Phi}(-\lambda)}{\partial \lambda} 
	= 
	- \frac{m(-\lambda) - \lambda m^{\prime}(-\lambda)}{(1- \gamma(1-\lambda m(-\lambda)))^2}
\end{align*}
This leads to the final form
\begin{align*}
    \frac{ \Theta^{\Phi}(-\lambda) + \lambda \frac{ \partial \Theta^{\Phi}(-\lambda)}{\partial \lambda } }{\big( 1 - \gamma ( 1 - \lambda m(-\lambda)) \big)^{2}}
    & = 
    \frac{ 1 - \lambda m(-\lambda)}{(1- \gamma(1-\lambda m(-\lambda)))^3} 
    - 
    \lambda \frac{m(-\lambda) - \lambda m^{\prime}(-\lambda)}{(1- \gamma(1-\lambda m(-\lambda)))^4} \\
    & = \frac{\gamma^{-1} (1-\lambda v(-\lambda))}{(\lambda v(-\lambda))^3} 
	- \lambda \frac{ \gamma^{-1}(v(-\lambda) - \lambda v^{\prime}(-\lambda))}{(\lambda v(-\lambda))^4}\\
    & = \frac{v^{\prime}(-\lambda)}{\gamma \lambda^2 v(-\lambda)^4} - \frac{1}{\gamma (\lambda v(-\lambda))^2}  
\end{align*}
where on the second equality we used \eqref{equ:STransformToComp}. Multiplying through by $\lambda^2$ then yields the quantity presented.

\subsubsection{Case: $\Phi(x) = 1$} 
The functional of interest in this case aligns with that calculated within \cite{dobriban2018high}, which we include below for completeness. In particular we have 
$\Theta^{\Phi}(-\lambda) = m(-\lambda)$ and as such we get 
\begin{align*}
\Theta^{\Phi}(-\lambda) + \lambda \frac{ \partial \Theta^{\Phi}(-\lambda)}{\partial \lambda } 
& = m (-\lambda) - \lambda m^{\prime}(-\lambda)  = \gamma^{-1}(v(-\lambda) - \lambda v^{\prime}(-\lambda))
\end{align*}
where on the second equality we used \eqref{equ:STransformToComp}. Dividing by $v(-\lambda)^2$ as well as adding the asymptotic variance we get, from Theorem \ref{main:thm}, the limit as $n,d \rightarrow \infty$ 
\begin{align*}
 \E_{\beta^{\star}}[
    \E_{\epsilon}[ R(\wh \beta_{\lambda})]
    - 
    R(\beta^{\star})
    ]
     & \rightarrow   
\sigma^2 \frac{1 - \lambda v(-\lambda)}{\lambda v(-\lambda)}
     - 
     \lambda \sigma^2 \frac{ v(-\lambda) - \lambda v^{\prime}(-\lambda)}{ (\lambda v(-\lambda))^2}
+ \frac{r^2}{\gamma} \frac{v(-\lambda) - \lambda v^{\prime}(-\lambda)}{v(-\lambda)^2}\\
& = \sigma^2 \Big( \frac{v^{\prime}(-\lambda)}{(v(-\lambda))^2} - 1\Big)
+ 
\frac{r^2}{\gamma v(-\lambda)} 
- 
\frac{r^2\lambda}{\gamma} \frac{v^{\prime}(-\lambda)}{v(-\lambda)^2} 
\end{align*}

\subsubsection{Case: $\Phi(x) = 1/x$}
The functional in the case $\Phi(x) = 1/x$ takes the form 
\begin{align*}
    \Theta^{\Phi}(z) 
    & = 
    \int \frac{1}{\tau} 
    \frac{1}{ \tau (1 - \gamma( 1 + z m(z))) - z }
    dH(\tau).
\end{align*}
Observe that we have 
\begin{align*}
    \int \frac{1}{\tau} dH(\tau) +z \Theta^{\Phi}(z) 
    & = 
    \int \frac{1}{\tau} \Big( 1 + \frac{z}{\tau(1-\gamma(1+z m(z))) - z} \Big) dH(\tau)\\
     & = 
    \int \frac{1}{\tau} 
    \frac{\tau (1 - \gamma( 1 + z m(z))) }{ \tau (1 - \gamma( 1 + z m(z))) - z }
    dH(\tau) \\
    & = 
    (1 - \gamma( 1 + z m(z))) 
    \int 
    \frac{1}{ \tau (1 - \gamma( 1 + z m(z))) - z }
    dH(\tau) \\
    & = 
    (1 - \gamma( 1 + z m(z)))  m(z).
\end{align*}
Solving for  $\Theta^{\Phi}(z) $ and plugging in the definition of the companion transform $v(z)$ we arrive at 
\begin{align*}
    \Theta^{\Phi}(z)  
    & = 
    \frac{1}{z} \big( 
    (1 - \gamma( 1 + z m(z)))  m(z) - \frac{1}{z} \int \frac{1}{\tau} dH(\tau)
    \big) \\
    & = 
    - v(z)
    \big( \frac{v(z)}{\gamma} + \frac{1}{z}(\frac{1}{\gamma} - 1 \big)\big)
    - 
    \frac{1}{z}\int \frac{1}{\tau} dH(\tau)\\
    & = 
    - \frac{ v(z)^2}{\gamma} 
    - \frac{v(z)}{z}(\frac{1}{\gamma} - 1 \big)
    - 
    \frac{1}{z} \int \frac{1}{\tau} dH(\tau).
\end{align*}
Fixing $z = -\lambda$ the quantity of interest then has the form
\begin{align*}
    \Theta^{\Phi}(-\lambda) 
    = 
    - \frac{v(-\lambda)^2}{\gamma} 
    + \frac{v(-\lambda)}{\lambda}(\frac{1}{\gamma} - 1 \big)
    + 
    \frac{1}{\lambda} 
    \int \frac{1}{\tau} dH(\tau),
\end{align*}
which when differentiated with respect to $\lambda$ yields 
\begin{align*}
    \frac{\partial \Theta^{\Phi}(-\lambda)  }{\partial \lambda }
    & =
    \frac{2v(-\lambda)v^{\prime}(-\lambda)}{\gamma} 
    - \frac{1}{\lambda}(\frac{1}{\gamma} - 1 \big)
    \big( \frac{v(-\lambda)}{\lambda} + v^{\prime}(-\lambda)\big)
    - \frac{1}{\lambda^2}\int \frac{1}{\tau} dH(\tau).
\end{align*}
Multiplying the above by $\lambda$ and adding $\Theta^{\Phi}(-\lambda)$ brings us to
\begin{align*}
    \Theta^{\Phi}(-\lambda) 
    + 
    \lambda 
    \frac{\partial \Theta^{\Phi}(-\lambda)  }{\partial \lambda }
    = 
    2 \lambda \frac{v^{\prime}(-\lambda)v(-\lambda) }{\gamma }
    - \frac{v(-\lambda)^2}{\gamma} 
    - (\frac{1}{\gamma} - 1 \big) v^{\prime}(-\lambda).
\end{align*}
Dividing the above by $v(-\lambda)^2$ and adding the limiting variance yields, from Theorem \ref{main:thm}, the limit as $n,d \rightarrow \infty$
\begin{align*}
    & \E_{\beta^{\star}}[
    \E_{\epsilon}[ R(\wh \beta_{\lambda})]
    - 
    R(\beta^{\star})
    ]\\
     & \rightarrow   
     \sigma^2 \frac{1 - \lambda v(-\lambda)}{\lambda v(-\lambda)}
     - 
     \lambda \sigma^2 \frac{ v(-\lambda) - \lambda v^{\prime}(-\lambda)}{ (\lambda v(-\lambda))^2}
     + 
     2  r^2 \lambda \frac{v^{\prime}(-\lambda)}{\gamma v(-\lambda)}
     - \frac{r^2}{ \gamma} - r^2 (\frac{1}{\gamma} - 1 \big) \frac{v^{\prime}(-\lambda)}{ v(-\lambda)^2}\\
     & = 
     \sigma^2 
     \big( \frac{v^{\prime}(-\lambda)}{(v(-\lambda))^2} - 1\big)
     + 
     2 r^2 \lambda \frac{v^{\prime}(-\lambda)}{\gamma v(-\lambda)}
     - \frac{r^2}{ \gamma} + 
      r^2 \lambda \frac{\gamma - 1}{\gamma} \frac{v^{\prime}(-\lambda)}{v(-\lambda)^2}
\end{align*}

\subsection{Strong and Weak Features Model }
\label{sec:ProofRR:TwoBulks}

This section presents the calculations associated to the strong and weak features model. We begin giving the stationary  point equation of the companion transform $v(t)$, after which we explicitly compute the limiting risk with the particular choice of $\Phi(x)$ in this case. Section \ref{sec:TwoBulks:RidgelessLimit:Proof} there after gives explicit form for the companion transform in the ridgeless limit.  Section \ref{sec:CorPreciseProof} gives the proof of Corollary \ref{Cor:PeciseOptimalReg} found within the main body of the manuscript. 

We begin by recalling the limiting spectrum of the covariance $\Sigma$ for the two Bulks Model is $dH(\tau) = \psi_1 \delta_{\rho_1} + \psi_2 \delta_{\rho_2}$. Recall we have $\psi_1 + \psi_2 = 1$ therefore we simply write $\psi_2 = 1-\psi_1$. 
Using the Silverstein equations \eqref{equ:CompanionTransform:StationaryPoint} the companion transform must satisfy 
\begin{align}
\label{equ:TwoBulks:SilverStein}
    \frac{-1}{v(t)}
    = 
    t - \gamma 
    \Big( \frac{\psi_1 \rho_1 }{1 + \rho_1 v(t)} + \frac{(1-\psi_1)\rho_2}{1 + \rho_2 v(t)}
    \Big),
\end{align}
meanwhile the derivative must satisfy
\begin{align}
\label{equ:CompanionDerivative}
    & \frac{1}{(v(t))^2}
    = 
    \frac{1}{v^{\prime}(t)} + \gamma 
    \Big( \frac{\psi_1 \rho_1^2 }{(1 + \rho_1 v(t))^2} + \frac{(1-\psi_1 )\rho_2^2}{(1 + \rho_2 v(t))^2}
    \Big) \\
    & \nonumber \implies 
    v^{\prime}(t) = \Big( 
    \frac{1}{(v(t))^2}
    - 
    \gamma 
    \Big( \frac{\psi_1 \rho_1^2 }{(1 + \rho_1 v(t))^2} + \frac{(1-\psi_1 )\rho_2^2}{(1 + \rho_2 v(t))^2}
    \Big)
    \Big)^{-1}
\end{align}
as such given $v(t)$ we can compute the derivative. Rearranging \eqref{equ:TwoBulks:SilverStein} and denoting $v(t) = v$ the companion transform evaluated at $t$ satisfies
\begin{align*}
    0 
    & = 
    (1+\rho_1 v)(1+\rho_2 v) + 
    t v(1+\rho_1 v)(1+\rho_2 v)
    - \gamma \psi_1\rho_1  v (1+\rho_2 v)
    - \gamma (1-\psi_1 ) \rho_2 v(1+\rho_1 v)\\
    & = 
    t \rho_1 \rho_2 v^3
    +
    ( t (\rho_1 + \rho_2) + (1-\gamma)\rho_1\rho_2) v^2
    +
    (t  + \rho_1 + \rho_2 - \gamma \psi_1 \rho_1 -\gamma (1-\psi_1 )\rho_2 )v 
    +1
\end{align*}
This cubic can then be solved computationally for different choices of $t$. In the case of the ridgeless limit $t \rightarrow 0$ in the overparameterised setting $\gamma > 1$, the above simplifies to a quadratic which can be solved, as shown in Section \ref{sec:TwoBulks:RidgelessLimit:Proof}. 

Now, recall in the strong and weak features model the structure of the ground truth $\beta^{\star}$ is such that $\Phi(x) = \phi_1 \mathbbm{1}_{x = \rho_1} + \phi_2 \mathbbm{1}_{x = \rho_2} $. To compute the limiting risk, specifically the bias, we must then evaluate $\Theta^{\Phi}(-\lambda)  + \lambda \frac{ \partial \Theta^{\Phi}(-\lambda)}{\partial \lambda }$. To this end, we have plugging $\Phi(x)$ into the definition of $\Theta^{\Phi}(z)$ 
\begin{align*}
    \Theta^{\Phi}(z) 
     & = 
    \int \Phi(\tau) 
    \frac{1}{ \tau (1 - \gamma( 1 + z m(z))) - z }
    dH(\tau)\\
     & = 
    \frac{\phi_1 \psi_1 }{  \rho_1(1 - \gamma( 1 + z m(z))) - z }
    +  \frac{\phi_2  (1- \psi_1) }{  \rho_2(1 - \gamma( 1 + z m(z))) - z }\\
    & = 
    \frac{\phi_1 \psi_1 }{-z(1 + \rho_1 v(z))}+     \frac{\phi_2 (1-\psi_1) }{-z(1 + \rho_2 v(z))}
\end{align*}
where on the last equality we used \eqref{equ:STransformToComp} to rewrite the above in terms of the companion transform. Plugging in the regularisation parameter $z=-\lambda$ we then get 
\begin{align*}
    \Theta^{\Phi}(-\lambda)
    = 
    \frac{\phi_1 \psi_1 }{\lambda(1+\rho_1 v(-\lambda))} + \frac{\phi_1 (1-\psi_1) }{\lambda(1+\rho_2 v(-\lambda))}.
\end{align*}
To the end of computing $\frac{ \partial \Theta^{\Phi}(-\lambda)}{\partial \lambda }$, we can differentiate the above to get
\begin{align*}
    \frac{ \partial \Theta^{\Phi}(-\lambda)}{\partial \lambda }
    = 
   - \phi_1 \psi_1  \frac{1 + \rho_1 v(-\lambda) - \lambda \rho_1 v^{\prime}(-\lambda)}
    {  \big( \lambda \rho_1 v(-\lambda) + \lambda \big)^2}
-  \phi_2 (1-\psi_1) \frac{1 + \rho_2 v(-\lambda) - \lambda \rho_2 v^{\prime}(-\lambda)}
    {  \big( \lambda \rho_2 v(-\lambda) + \lambda \big)^2},
\end{align*}
 which yields
\begin{align*}
    \Theta^{\Phi}(-\lambda)  
    + 
    \lambda \frac{ \partial \Theta^{\Phi}(-\lambda)}{\partial \lambda }
    = 
    \phi_1 \psi_1  \frac{ \rho_1 v^{\prime}(-\lambda)}{(\rho_1 v(-\lambda) + 1)^2} + \phi_2 (1-\psi_1) \frac{ \rho_2 v^{\prime}(-\lambda)}{(\rho_2 v(-\lambda) + 1)^2}
\end{align*}
as required. The final form for the limiting risk is then 
\begin{align*}
    & \lim_{n,d \rightarrow \infty} 
    \E_{\beta^{\star}}[
    \E_{\epsilon}[ R(\wh \beta_{\lambda})]
    - 
    R(\beta^{\star})
    ]\\
    & =
     \sigma^2 \frac{1 - \lambda v(-\lambda)}{\lambda v(-\lambda)}
     - 
     \lambda \sigma^2 \frac{ v(-\lambda) - \lambda v^{\prime}(-\lambda)}{ (\lambda v(-\lambda))^2}
     + 
    r^2  \sum_{i=1}^{2} \phi_i \psi_i \frac{ \rho_i v^{\prime}(-\lambda)}{(\rho_i v(-\lambda) + 1)^2 v(-\lambda)^2}
    \\
    & = 
    - \sigma^2 
    + 
    \sigma^2 \frac{v^{\prime}(-\lambda)}{(v(-\lambda))^2}
    + r^2  \sum_{i=1}^{2} \phi_i \psi_i \frac{\rho_i v^{\prime}(-\lambda)}{( v(-\lambda)^2 (\rho_i v(-\lambda) + 1)^2}.
\end{align*}

\subsubsection{Ridgeless Limit}
\label{sec:TwoBulks:RidgelessLimit:Proof}
To consider the Ridgeless limit $t\rightarrow 0$ of the companion transform $v(t)$, some care must be taken about which regime $\gamma <1$ or $\gamma >1$ we are in.
\paragraph{Underparameterised $\gamma < 1$}
Following  the proof of Lemma 6.2 in \cite{dobriban2018high} we  have in the underparameterised case $\gamma <1$ the limit $\lim_{t \rightarrow 0_{-}} t v(t) =  1-\gamma$. 

\paragraph{Overparameterised $\gamma > 1$} 
Following the proof of Lemma 6.2 in \cite{dobriban2018high} when $\gamma >1$ we have the limit $\lim_{t \rightarrow 0_{-}} v(t) = v(0)$.
From dominated convergence theorem we can take the limit in the Silverstein equation \eqref{equ:TwoBulks:SilverStein} to arrive at the quadratic
\begin{align*}
    0 
    & = 
    (1+\rho_1 v)(1+\rho_2 v)
    - \gamma \psi_1 \rho_1  v (1+\rho_2 v)
    - \gamma (1-\psi_1) \rho_2 v(1+\rho_1 v) \\
    & = 
    (1-\gamma)\rho_1\rho_2 v^2
    +
    (\rho_1 + \rho_2 - \gamma \psi_1 \rho_1 -\gamma (1-\psi_1)\rho_2 )v 
    +1
\end{align*}
Solving for $v$ with the quadratic formula immediately gives
\begin{align}
\label{equ:TwoBulks:CompanionTransform:Ridgeless}
      v(0) 
    = 
    \frac{ - (\rho_1 + \rho_2 - \gamma \psi_1 \rho_1 -\gamma (1-\psi_1 )\rho_2 ) - 
    \sqrt{ (\rho_1 + \rho_2 - \gamma \psi_1\rho_1 -\gamma (1-\psi_1)\rho_2 )^2 - 4 (1 - \gamma ) \rho_1 \rho_2 }}{2(1-\gamma)\rho_1\rho_2  }
    .
\end{align}
Recall from \cite{silverstein1995analysis} we have that $v(z) \in \mathcal{S}$, as such we take the sign above which yields a non-negative quantity. Noting we we focus on the regime where $\gamma > 1$, we see for the above to be non-negative we require the numerator to be negative, and thus, we take the negative sign.

\subsubsection{Proof of Corollary \ref{Cor:PeciseOptimalReg}}
\label{sec:CorPreciseProof}
In this section we provide the proof of Corollary \ref{Cor:PeciseOptimalReg}. The proof essentially requires computing the companion transform in this case and checking the sign of the asymptotic derivative at zero i.e. $R^{\prime}_{\mathrm{Asym}}(0)$.  Let us begin by noting that when $\gamma =2$ and $\psi_1 = \psi_2 = 1/2$ that the companion transform at zero is $v(0) = 1/\sqrt{\rho_1 \rho_2}$. Let us now compute quantities related to both the first derivative $v^{\prime}(0)$ and second derivative $v^{\prime\prime}(0)$. Using  \eqref{equ:CompanionDerivative}, we can, by dividing both sides by $v(t)^2$ and taking $t \rightarrow 0$, get
\begin{align*}
\frac{v^{\prime}(0) }{v(0)^2} =  \Big( 1 - \frac{\rho_1 + \rho_2}{ (\sqrt{\rho_1}  + \sqrt{\rho_2})^2} \Big)^{-1} 
= \frac{1}{2}\big( \sqrt{\frac{\rho_1}{\rho_2}} + \sqrt{\frac{\rho_2}{\rho_1}}\big) + 1
\end{align*}
Meanwhile, recall by differentiating both sides of the silverstein equations \eqref{equ:CompanionTransform:StationaryPoint} in $t$ we can get
\begin{align*}
 \frac{v^{\prime}(t)}{v(t)^2}
 = 
 1 + \gamma v^{\prime}(t) \int \frac{\tau^2}{(1 + \tau v(t))^2} dH(\tau).
 \end{align*}
 Therefore, if we differentiate once more we get
 \begin{align*}
 	\frac{v^{\prime\prime} (t) }{v(t)^2 } - 2 \frac{v^{\prime}(t)^2}{(v(t))^3} = \gamma v^{\prime\prime}(t) \int \frac{\tau^2}{(1 + \tau v(t))^2} dH(\tau)
 	 - 2 \gamma v^{\prime}(t)^2 \int \frac{\tau^3 }{(1+\tau v(t))^3} d H(\tau),
 \end{align*}
 and thus, multiplying through by $v(t)^3/v^{\prime}(t)^2$ and rearranging we arrive at
 \begin{align*}
 	\frac{v^{\prime\prime}(t) v(t) }{v^{\prime}(t)^2  } 
 	\Big[ 1 - \gamma \int \frac{\tau^2 v(z)^2 }{(1 + \tau v(z))^2} dH(\tau)\Big]
 	= 
 	2 \Big[ 1 - \gamma \int \frac{\tau^3 v(z)^3}{(1 + \tau v(z))^3} d H(\tau) \Big].
 \end{align*}
Furthermore, noting that $1 - \gamma \int \frac{\tau^2 v(z)^2 }{(1 + \tau v(z))^2} dH(\tau) = \frac{v^{\prime}(t)}{v(t)^2}$ means we get the following equality for the second derivative
 \begin{align*}
 v^{\prime\prime}(t) 
 = 
 2 \Big[ 1 - \gamma \int \frac{\tau^3 v(t)^3}{(1 + \tau v(t))^3} d H(\tau) \Big]
 \Big( \frac{v^{\prime}(t)}{v(t)} \Big)^3.
 \end{align*}
 Taking $t \rightarrow 0$ and plugging in the defintion of $v(0)$ yields the following, which will be required for the proof
 \begin{align*}
 v^{\prime\prime}(0) = 
 2 \Big[ 1 - \frac{\rho_1^{3/2} + \rho_2^{3/2}}{(\sqrt{\rho_1} + \sqrt{\rho_2})^3} \Big]
 \Big( \frac{v^{\prime}(0)}{v(0)} \Big)^3.
 \end{align*}
 
 Now, let us compute the derivative of the asymmptotic risk $R_{\mathrm{Asymm}}(\lambda)$ for the strong and weak features model. Bringing together the Bias and Variance terms, differentiating through by $\lambda$ and dividing by $\sigma^2$ we get
 \begin{align*}
\frac{1}{\sigma^2} 
R^{\prime}_{\mathrm{Asym}}(\lambda)
&  = 
\Big(  2 \frac{(v^{\prime}(-\lambda))^2 }{(v(-\lambda))^3}  - \frac{ v^{\prime\prime}(-\lambda)}{(v(-\lambda)^2)} \Big)  
\Big( 1 + \frac{r^2}{\sigma^2} \sum_{i=1}^{2}\frac{\phi_i \psi_i \rho_i}{(\rho_i v(-\lambda) + 1)^2}
\Big)\\
&\quad\quad 
+ 2 
\Big(\frac{v^{\prime}(-\lambda)}{v(-\lambda) }
\Big)^2 
\frac{r^2}{\sigma^2} \sum_{i=1}^{2}\frac{\phi_i \psi_i \rho_i^2 }{(\rho_i v(-\lambda) + 1)^3} 
 \end{align*}
 Taking $\lambda \rightarrow 0$ 
 and plugging in $\psi_1 = \psi_2 = 1/2$, $\psi_1 \phi_1 + \psi_2 \phi_2 = 1$, $\phi_1 + \phi_2 = 2$ as well as $v(0)$ into $\rho_i v(0)$ for $i=1,2$ yields
 \begin{align*}
\frac{1}{\sigma^2} 
R^{\prime}_{\mathrm{Asym}}(0)
 & = 
\Big(  2 \frac{v^{\prime}(0)^2 }{v(0 )^3}  - \frac{ v^{\prime\prime}(0 )}{v(0 )^2)} \Big)  
\Big( 1 + \frac{r^2}{\sigma^2} \frac{ \rho_1 \rho_2 }{(\sqrt{\rho_1} + \sqrt{\rho_2} )^2}
\Big)
+ 
\Big(\frac{v^{\prime}(0)}{v(0) }
\Big)^2 
\frac{r^2}{\sigma^2}( \rho_1 \rho_2)^{3/2} \frac{\phi_1 \sqrt{\rho_1} + \phi_2 \sqrt{\rho_2}}{(\sqrt{\rho_1} + \sqrt{\rho_2})^3 } .
 \end{align*}
 Let us now plug in the second derivative $v^{\prime\prime}(0)$. In particular, note that we can write 
 
\begin{align*}
2\frac{v^{\prime}(0)^2 }{v(0)^3} - \frac{ v^{\prime\prime}(0)}{v(0)^2}
& = 
2 \big(\frac{v^{\prime}(0)}{v(0)}\big)^2 
\frac{1}{v(0)} 
\Big(  1 - \Big( 1 - \frac{\rho_1^{3/2} + \rho_2^{3/2} } {(\sqrt{\rho_1} + \sqrt{\rho_2})^3} \Big)  \frac{v^{\prime}(0)}{v(0)^2} 
\Big) \\
& = 
- \Big(\frac{v^{\prime}(0)}{v(0)}\Big)^2 
\frac{1}{v(0)} 
\end{align*}
where on the second equality we have used the equality for $\frac{v^{\prime}(0)}{v(0)^2}$ from above to note that
\begin{align*}
& 1  - \Big( 1 - \frac{\rho_1^{3/2} + \rho_2^{3/2} } {(\sqrt{\rho_1} + \sqrt{\rho_2})^3} \Big)  \frac{v^{\prime}(0)}{v(0)^2} \\
& = 
\frac{\rho_1^{3/2} + \rho_2^{3/2} }{(\sqrt{\rho_1} + \sqrt{\rho_2})^3} - \frac{1}{2} \sqrt{\frac{\rho_1}{\rho_2}} - \frac{1}{2} \sqrt{\frac{\rho_2}{\rho_1}} 
+
\frac{\rho_1^{3/2} + \rho_2^{3/2} }{(\sqrt{\rho_1} + \sqrt{\rho_2})^3} \big( \sqrt{\frac{\rho_1}{\rho_2}} + \sqrt{\frac{\rho_2}{\rho_1}}) \frac{1}{2} \\
& = 
\frac{ (\rho_1^{3/2} + \rho_2^{3/2} )( 2 \sqrt{\rho_1 \rho_2} + \rho_1 + \rho_2) - (\rho_1 + \rho_2)(\sqrt{\rho_1} + \sqrt{\rho_2})^3 }{2 \sqrt{\rho_1 \rho_2}(\sqrt{\rho_1} + \sqrt{\rho_2})^3}\\
& = 
\frac{ (\rho_1^{3/2} + \rho_2^{3/2} )( 2 \sqrt{\rho_1 \rho_2} + \rho_1 + \rho_2) - (\rho_1 + \rho_2)(\sqrt{\rho_1} + \sqrt{\rho_2})^3 }{2 \sqrt{\rho_1 \rho_2}(\sqrt{\rho_1} + \sqrt{\rho_2})^3}\\
& = 
\frac{ (\rho_1^{3/2} + \rho_2^{3/2} ) - (\rho_1 + \rho_2)(\sqrt{\rho_1} + \sqrt{\rho_2})}{2 \sqrt{\rho_1 \rho_2}(\sqrt{\rho_1} + \sqrt{\rho_2})}\\
& = 
- \frac{1}{2}.
\end{align*}
Returning to the derivative of the asymptotic risk $R^{\prime}_{\mathrm{Asym}}(0)$, factoring out $\Big(\frac{v^{\prime}(0)}{v(0)}\Big)^2 $ and plugging in the definition of $v(0)$ gives 
\begin{align*}
\frac{1}{\sigma^2} 
R^{\prime}_{\mathrm{Asym}}(0)
=
\Big(\frac{v^{\prime}(0)}{v(0)}\Big)^2 \sqrt{\rho_1 \rho_2} 
\Big[ 
-1 
+ 
\frac{r^2}{\sigma^2} \frac{\rho_1 \rho_2}{(\sqrt{\rho_1} + \sqrt{\rho_2})^2} 
\Big(
\frac{\phi_1 \rho_1 + \phi_2 \rho_2}{\sqrt{\rho_1} +\sqrt{\rho_2} }  - 1 \Big)
\Big]
\end{align*}
It is then clear that the sign of $R^{\prime}_{\mathrm{Asym}}(0)$ is governed by the quantity in the square brackets. This then yields the result.

\section{Proof of Lemma \ref{lem:Limits}}
\label{app:proof-lemma-rmt}
In this section we provide the proof for Lemma \ref{lem:Limits}. We recall that the limits \eqref{equ:lem:limits:1} and \eqref{equ:lem:limits:2} have been computed previously. 
In particular, Lemma  2.2 of \cite{ledoit2011eigenvectors} (the roles of $d,n$ are swapped in their work, and thus, one must swap $\gamma$ with $1/\gamma$) shows 
\begin{align*}
    \frac{1}{d}\trace\Big( \Big( \frac{X^{\top}X}{n} + \lambda I \Big)^{-1} \Sigma \Big) 
    \rightarrow 
    \gamma^{-1}\Big( \frac{1}{1 - \gamma(1-\lambda m(-\lambda)) } - 1 \Big)
\end{align*}
Meanwhile  Lemma 7.4 of \cite{dobriban2018high} shows 
\begin{align*}
    \frac{1}{d}\trace\Big( \Big( \frac{X^{\top}X}{n} + \lambda I \Big)^{-2} \Sigma \Big) \rightarrow 
    \frac{m(-\lambda) - \lambda m^{\prime}(-\lambda)}{(1-\gamma(1-\lambda m(-\lambda)))^2}
\end{align*}
This leaves us to show the limit \eqref{equ:lem:limits:3}, for which we build upon the techniques \cite{ledoit2011eigenvectors} as well as \cite{chen2011regularized}. 

We begin with the decomposition. Recall since the covariates are multivariate Gaussians, they can be rewritten as $X = Z \Sigma^{1/2}$ where $Z \in \R^{n \times d}$ is a matrix of independent standard normal Gaussian random variables. For $i =1,\dots,n$ the associated row in $X$ is then denoted $X_i = Z_i \Sigma^{1/2}$. As such $X^{\top} X = \sum_{i=1}^{n} X_{i}^{\top} X_i = \sum_{i=1}^{n} \Sigma^{1/2} Z_i^{\top} Z_i \Sigma^{1/2}$. Let us then define $R_{i}(z) = \big( \frac{X^{\top} X}{n} - \frac{X_i^{\top} X_i}{n}  -z I \big)^{-1}$. Using the Sherman-Morrison formula we then get 
\begin{align}
\label{equ:ShermanMorr}
    R(z) = (\frac{X^{\top}X}{n} - z I)^{-1}
    = 
    R_{i}(z) 
    - 
    \frac{1}{n} \frac{ R_{i}(z)\Sigma^{1/2} Z_i^{\top} Z_{i} \Sigma^{1/2} R_{i}(z)}{1 
    + \frac{1}{n} Z_i \Sigma^{1/2} R_{i}(z) \Sigma^{1/2} 
    Z_{i}^{\top} }
\end{align}
Moreover we have
\begin{align*}
   \frac{1}{n} \sum_{i=1}^{n} \Sigma^{1/2} Z_i^{\top} Z_{i} \Sigma^{1/2}R(z)
   =
   \frac{X^{\top} X}{n} R(z) = (\frac{X^{\top} X}{n} - zI) R(z) + z R(z) = I + z R(z) 
\end{align*}
Multiplying the above on the left by $\Phi(\Sigma)R(z)$, taking the trace and dividing by $d$ yields  
\begin{align*}
    \frac{1}{d} 
    \trace\big( \Phi(\Sigma)R(z)\big) 
    + 
    z \frac{1}{d} 
    \trace\big( \Phi(\Sigma)R(z)^2\big) 
   &  = 
    \frac{1}{d} \sum_{i=1}^{n}  \frac{1}{n} 
    Z_{i} \Sigma^{1/2}R(z)\Phi(\Sigma)R(z)\Sigma^{1/2}Z_i^{\top} \\
    & = 
    \frac{1}{d} \sum_{i=1}^{n} 
    \frac{1}{n} \frac{ Z_{i} \Sigma^{1/2}R_i(z)\Phi(\Sigma)R_i(z)\Sigma^{1/2}Z_i^{\top}}{
    \big( 1  + \frac{1}{n} Z_i \Sigma^{1/2} R_{i}(z) \Sigma^{1/2}Z_{i}^{\top}\big)^2}
\end{align*}
where for $i=1,\dots,n$ we have plugged in \eqref{equ:ShermanMorr}  twice into for $R(z)$ to get
\begin{align*}
     & Z_{i} \Sigma^{1/2}R(z)\Phi(\Sigma)R(z)\Sigma^{1/2}Z_i^{\top}\\
     & = 
    Z_{i} \Sigma^{1/2}R_i(z)\Phi(\Sigma)R(z)\Sigma^{1/2}Z_i^{\top}
    - 
    \frac{1}{n} 
    \frac{Z_{i} \Sigma^{1/2}R_{i}(z)\Sigma^{1/2} Z_i^{\top} Z_{i} \Sigma^{1/2} R_{i}(z)\Phi(\Sigma)R(z)\Sigma^{1/2}Z_i^{\top}}{1 
    + \frac{1}{n} Z_i \Sigma^{1/2} R_{i}(z) \Sigma^{1/2}Z_{i}^{\top} } \\
    & =
    \frac{ Z_{i} \Sigma^{1/2}R_i(z)\Phi(\Sigma)R(z)\Sigma^{1/2}Z_i^{\top}}{1 
    + \frac{1}{n} Z_i \Sigma^{1/2} R_{i}(z) \Sigma^{1/2}Z_{i}^{\top} } \\
    & = 
    \frac{1}{ 1  + \frac{1}{n} Z_i \Sigma^{1/2} R_{i}(z) \Sigma^{1/2}Z_{i}^{\top} } \\
    & \quad \times \Big[ 
    Z_{i} \Sigma^{1/2}R_i(z)\Phi(\Sigma)R_i(z)\Sigma^{1/2}Z_i^{\top}
    - 
    \frac{1}{n}
    \frac{ Z_{i} \Sigma^{1/2}R_i(z)\Phi(\Sigma) R_{i}(z)\Sigma^{1/2} Z_i^{\top} Z_{i} \Sigma^{1/2} R_{i}(z) 
    \Sigma^{1/2}Z_i^{\top}}{  1 + \frac{1}{n} Z_i \Sigma^{1/2} R_{i}(z) \Sigma^{1/2}Z_{i}^{\top}}
    \Big]\\
    & = 
    \frac{ Z_{i} \Sigma^{1/2}R_i(z)\Phi(\Sigma)R_i(z)\Sigma^{1/2}Z_i^{\top}}{
    \big( 1  + \frac{1}{n} Z_i \Sigma^{1/2} R_{i}(z) \Sigma^{1/2}Z_{i}^{\top}\big)^2}.
\end{align*}
Choosing $z = -\lambda$ we then have that
\begin{align}
\label{equ:LimitObject}
    \frac{1}{d} 
    \trace\big( \Phi(\Sigma)R(-\lambda)\big) 
    - 
    \lambda \frac{1}{d} 
    \trace\big( \Phi(\Sigma)R(-\lambda)^2\big)  
    = 
    \frac{1}{d} \sum_{i=1}^{n} 
    \frac{ \frac{1}{n} \trace\big(\Sigma^{1/2} R(-\lambda) \Phi(\Sigma) R(-\lambda) \Sigma^{1/2}\big)}{
    \big( 1 + \frac{1}{n} \trace(\Sigma R(-\lambda) ) \big)^2 }
    + 
    \delta
\end{align}
where the error term $\delta = \delta_1 + \delta_2 + \delta_3 + \delta_4$ such that 
\begin{align*}
    & \delta_1  = 
    \frac{1}{d} \sum_{i=1}^{n} 
    \frac{ \frac{1}{n} \trace\big(\Sigma^{1/2} R_i(-\lambda) \Phi(\Sigma) R_i(-\lambda) \Sigma^{1/2}\big) 
    - \frac{1}{n} \trace\big(\Sigma^{1/2} R(-\lambda) \Phi(\Sigma) R(-\lambda) \Sigma^{1/2}\big)
    }{\big( 1 + \frac{1}{n} \trace(\Sigma R(-\lambda) ) \big)^2 }
    \\
    & \delta_2 = 
    \frac{1}{d} \sum_{i=1}^{n} 
    \frac{1}{n} \trace\big(\Sigma^{1/2} R_i(-\lambda) \Phi(\Sigma) R_i(-\lambda) \Sigma^{1/2}\big) 
    \Big( 
    \frac{ 1 }{\big( 1 + \frac{1}{n} \trace(\Sigma R_i(-\lambda) ) \big)^2 }
    - 
    \frac{ 1 }{\big( 1 + \frac{1}{n} \trace(\Sigma R(-\lambda) ) \big)^2 }
    \Big)\\
    & \delta_3 = 
    \frac{1}{d} \sum_{i=1}^{n} 
    \frac{1}{n} \trace\big(\Sigma^{1/2} R_i(-\lambda) \Phi(\Sigma) R_i(-\lambda) \Sigma^{1/2}\big) \\
    & \quad\quad\quad\quad\quad\quad \times 
    \Big( 
    \frac{ 1 }{\big( 1  + \frac{1}{n} Z_i \Sigma^{1/2} R_{i}(-\lambda) \Sigma^{1/2}Z_{i}^{\top}\big)^2 }
    - 
    \frac{ 1 }{\big( 1 + \frac{1}{n} \trace(\Sigma R_i(-\lambda) ) \big)^2 } 
    \Big)\\
    & \delta_4
    =
    \frac{1}{d} \sum_{i=1}^{n} 
     \frac{ \frac{1}{n} Z_{i} \Sigma^{1/2}R_i(-\lambda)\Phi(\Sigma)R_i(-\lambda)\Sigma^{1/2}Z_i^{\top} 
     - 
     \frac{1}{n} \trace\big(\Sigma^{1/2} R_i(-\lambda) \Phi(\Sigma) R_i(-\lambda) \Sigma^{1/2}\big)
     }{
    \big( 1  + \frac{1}{n} Z_i \Sigma^{1/2} R_{i}(z) \Sigma^{1/2}Z_{i}^{\top}\big)^2}
\end{align*}
As shown in section \ref{sec:ErrorTermsConvergence} the error terms $|\delta_1|,|\delta_2|,|\delta_3|,|\delta_4| \rightarrow 0$ almost surely as $n,d \rightarrow \infty$. It is now a matter of computing the limits of the remaining terms. As discussed previously the limit of $\frac{1}{d} \trace(\Sigma R(-\lambda))$ is known from \cite{ledoit2011eigenvectors}. From the same work it is also known that
\begin{align}
\label{equ:limit:phi}
    \frac{1}{d}\trace\big( \Phi(\Sigma) R(-\lambda) \big) 
    \rightarrow \Theta^{\Phi}(-\lambda).
\end{align}
That leaves us to compute the limit of $\frac{1}{d} \trace\big(\Phi(\Sigma) R(-\lambda)^2 \big)$.
If we are to write $f_{d}(\lambda) = \frac{1}{d}\trace\big( \Phi(\Sigma) R(-\lambda) \big) $ then note the derivative with respect to $\lambda$ is $f^{\prime}_{d}(\lambda) =  - \frac{1}{d} \trace\big(\Phi(\Sigma) R(-\lambda)^2 \big)$. We wish to now study the limit of the $f^{\prime}_{d}(\lambda)$ through the limit of $f_{d}(\lambda)$. To do so we will follow the steps in \cite{dobriban2018high}, which will require some definitions and the following theorem.

Let $D$ be a domain, i.e. a connected open set of $\mathbb{C}$. A function $f: D \rightarrow \mathbb{C}$ is called analytic on $D$ if it is differentiable as a function of the complex variable $z$ on $D$. The following key theorem, sometimes known as Vitali's Theorem, ensures that the derivatives of converging analytic functions also converge.  
\begin{theorem}[Lemma 2.14 in \cite{bai2010spectral}]
\label{thm:Vitalis}
Let $f_1,f_2,\dots$ be analytic on the domain $D$, satisfying $|f_n(z)| \leq M$ for every $n$ and $z$ in $D$. Suppose that there is an analytic function $f$ on $D$ such that $f_{n}(z) \rightarrow f(z)$ for all $z \in D$. Then it also holds that $f^{\prime}_{n}(z) \rightarrow f^{\prime}(z)$ for all $z \in D$ 
\end{theorem}
Now we have from \cite{ledoit2011eigenvectors}
\begin{align*}
    f_{d}(\lambda) 
    \rightarrow 
    \int \Phi(\tau) \frac{1}{\tau (1 - \gamma(1-\lambda m(-\lambda))) + \lambda} d H(\tau)
\end{align*}
for all $\lambda \in \mathcal{S} : = \{u + i v: v \not= 0, \text{ or } v = 0, u > 0\}$. Checking the conditions of Theorem \ref{thm:Vitalis} we have that $f_{d}(\lambda)$ is an analytic function of $\lambda$ on $\mathcal{S}$ and is bounded $|f_d(\lambda)| \leq \frac{\|\Phi(\Sigma)\|_2}{\lambda}$. To apply Theorem \ref{thm:Vitalis} it suffices to show that the limit $\Theta^{\Phi}(-\lambda)$ is analytical. To this end we invoke Morera's theorem which states if 
\begin{align*}
    \oint_{\gamma} \Theta^{\Phi}(-\lambda) d \lambda  = 0
\end{align*}
for any closed curve $\gamma$ in the region $\mathcal{S}$ then $\Theta^{\Phi}(-\lambda)$ is analytic. We see this is the case by applying Fubini's Theorem as follows 
\begin{align*}
    \oint_{\gamma} \Theta^{\Phi}(-\lambda) d \lambda 
    & = 
    \oint_{\gamma} \int \Phi(\tau) \frac{1}{\tau (1 - \gamma(1-\lambda m(-\lambda))) + \lambda} d H(\tau) d\lambda \\
    & = 
    \int \Phi(\tau) \underbrace{ 
    \oint_{\gamma}   \frac{1}{\tau (1 - \gamma(1-\lambda m(-\lambda))) + \lambda} d\lambda}_{ = 0} d H(\tau)
    = 0
\end{align*}
and noting that the inner integral is zero from Cauchy Theorem  as $\frac{1}{\tau (1 - \gamma(1-\lambda m(-\lambda))) + \lambda}$ is an analytical function of $\lambda$ in $\mathcal{S}$ for any $\tau \in [h_1,h_2]$. By Theorem \ref{thm:Vitalis} we have that 
\begin{align}
\label{equ:limit:phiprime}
    - \frac{1}{d} \trace\big(\Phi(\Sigma) R(-\lambda)^2 \big) 
    = f^{\prime}_{d}(-\lambda) \rightarrow \frac{\partial \Theta^{\Phi}(-\lambda)}{\partial \lambda}.
\end{align} 
The  final limit \eqref{equ:lem:limits:3} is arrived at by considering the limit as $d,n\rightarrow \infty$ of \eqref{equ:LimitObject}. Specifically, with the fact that $\delta \rightarrow 0$, bringing together \eqref{equ:limit:phi}, \eqref{equ:limit:phiprime} and \eqref{equ:lem:limits:1}. Noting that \eqref{equ:lem:limits:1} is applied to the square of $1 + \frac{1}{n} \trace\big( \Sigma R(-\lambda)\big) = 1 + \gamma \frac{1}{d} \trace\big( \Sigma R(-\lambda)\big) \rightarrow \frac{1}{1 - \gamma(1-\lambda m(-\lambda))}$.

\subsection{Showing $\delta \rightarrow 0$}
\label{sec:ErrorTermsConvergence}
To analyse these quantities we introduce the following concentration inequality from Lemma A.2 of \cite{paul2007asymptotics} with $\delta = 1/3$.
\begin{lemma}
\label{lem:Conc}
Suppose $y$ is $d-$dimensional Gaussian random vector $y \sim \mathcal{N}(0,I)$ and $C \in \R^{d \times d}$ is a symmetric matrix such that $\| C \| \leq L$. Then for all $0 < t < L$, 
\begin{align*}
    \P\big( \frac{1}{d} | y C y^{\top} - \trace(C)| > t\big) \leq 
    2 \exp
    \Big\{
    - \frac{p t^2}{6 L^2}
    \Big\}.
\end{align*}
\end{lemma}
Furthermore, we will use the fact that the maximal eigenvalues are upper bounded  
\begin{align*}
    \|R(-\lambda)\|_2 \leq \frac{1}{\lambda}
    \quad 
    \text{ and } 
    \max_{1 \leq i \leq n }\|R_{i}(-\lambda)\|_2 \leq \frac{1}{\lambda}
\end{align*}
We proceed to show that each of the error $\delta_1,\delta_2,\delta_3,\delta_4$ converge to zero almost surely. 

Begin with $\delta_1$. For $i = 1,\dots, n$  by adding and subtracting $\trace(\Sigma^{1/2} R(z) \Phi(\Sigma) R_i(z) \Sigma^{1/2})$ we can decompose 
\begin{align*}
    & \trace\big(\Sigma^{1/2} R_i(-\lambda) \Phi(\Sigma) R_i(-\lambda) \Sigma^{1/2}\big) -  \trace\big(\Sigma^{1/2} R(z) \Phi(\Sigma) R(z) \Sigma^{1/2}\big) \\
    & = 
    \trace\big( (R_i(-\lambda) - R(-\lambda))\Phi(\Sigma)R_i(-\lambda) \Sigma \big)
    + 
    \trace\big(\Sigma R(-\lambda) \Phi(\Sigma)( R_{i}(-\lambda) - R(-\lambda)) \big)
\end{align*}
Using \eqref{equ:ShermanMorr} and letting $A = \Phi(\Sigma)R_i(-\lambda) \Sigma$ we then get 
\begin{align}
\label{equ:BoundLeaveOneOut}
    \frac{1}{n} 
    |\trace\big( (R_i(-\lambda) - R(-\lambda))A \big)|
    & = 
    \Big|  
    \frac{1}{n^2} 
    \frac{ Z_i \Sigma^{1/2}R_i(-\lambda) A R_i(-\lambda) \Sigma^{1/2}Z_i^{\top}}{1 + \frac{1}{n} Z_i \Sigma^{1/2}R_i(-\lambda)\Sigma^{1/2}Z_i^{\top}}
    \Big|\\
    & \leq 
    \frac{\|A\|_2}{n} 
    \Big| 
    \frac{1}{n} \frac{ Z_i \Sigma^{1/2}R_i(-\lambda) R_i(-\lambda) \Sigma^{1/2}Z_i^{\top}}{1 + \frac{1}{n} Z_i \Sigma^{1/2}R_i(-\lambda)\Sigma^{1/2}Z_i^{\top}}
    \Big|
    \nonumber 
    \\
    & \leq 
    \frac{\|A\|_2}{n} 
    \sup_{x} 
    \Big|
    \frac{ x R_{i}(-\lambda)^2 x^{\top}}{1 + x R_{i}(-\lambda) x^{\top}}
    \Big|
    \nonumber
    \\
    & \leq 
    \frac{\|A\|_2}{n}
    \sup_{x} 
    \Big|
    \frac{ x R_{i}(-\lambda)^2 x^{\top}}{x R_{i}(-\lambda) x^{\top}}
    \Big|
    \nonumber
    \\ 
    & \leq \frac{\|A\|_2}{n} \|R_i(-\lambda)\|_2
    \nonumber
    \\
    & \leq \frac{\|A\|_2}{\lambda n}
    \nonumber
    \\
    & \leq \frac{ \|\Phi(\Sigma)\|_2 \|\Sigma\|_2 }{\lambda^2 n}
    \nonumber
\end{align}

An identical calculation with $A = \Phi(\Sigma)R(-\lambda) \Sigma$ yields the same bound.  This then yields with the lower bound $(1 + \trace(\Sigma^{1/2} R(-\lambda) \Sigma^{1/2})) \geq 1$ 
\begin{align*}
    |\delta_1| \leq 2 \frac{n}{d}  \frac{ \|\Phi(\Sigma)\|_2 \|\Sigma\|_2 }{\lambda^2 n}
\end{align*}
and as such $\delta_1$ goes to zero as $n,d\rightarrow \infty$ so that $d/n \rightarrow  \gamma$.

Now consider the term $\delta_2$. Note that for two positive numbers $a,b \geq 0$ we have 
\begin{align*}
    \frac{1}{(1+a)^{2}} - \frac{1}{(1+b)^{2}} 
    & = 
     \frac{(1 + b)^2 - (1+a)^2}{(1+a)^2 (1+b)^2}\\
    & = 
    \frac{b^2 + 2b - a^2 - 2a}{(1+a)^2 (1+b)^2}\\
    & =  
    \frac{ b(b-a) + a(b-a) + 2(b-a)}{(1+a)^2 (1+b)^2}\\
    & = 
    (b-a) \frac{ (b+1) + (a+1)}{(1+a)^2 (1+b)^2}\\
    & = 
    (b-a)\big( \frac{1}{ (1+a)^2(1+b) } + \frac{1}{ (1+a)(1+b)^2 }\big)
\end{align*} 
and as such $| (1+a)^{-2} - (1+b)^{-2}| \leq 2|b-a|$. Using this with $a = \frac{1}{n} \trace(\Sigma^{1/2}R_i(-\lambda)\Sigma^{1/2})$ and $b = \frac{1}{n} \trace(\Sigma^{1/2} R(-\lambda)\Sigma^{1/2} )$ whom are both non-negative, allows us to upper bound 
\begin{align*}
    \Big| 
    \frac{ 1 }{\big( 1 + \frac{1}{n} \trace(\Sigma R_i(-\lambda) ) \big)^2 }
    - 
    \frac{ 1 }{\big( 1 + \frac{1}{n} \trace(\Sigma R(-\lambda) ) \big)^2 }
    \Big|
    & \leq 
    2 \frac{1}{n}\big| \trace(\Sigma R_i(-\lambda) ) -  \trace(\Sigma R(-\lambda) )\big|\\
    & \leq 
    \frac{2 \|\Sigma\|}{\lambda n}
\end{align*}
where for the final inequality we used the argument \eqref{equ:BoundLeaveOneOut} with $A = \Sigma$. Now, since the eigenvalues in the following trace are non-negative we can upper bound 
\begin{align}
    \frac{1}{d} \big| \trace\big(\Sigma^{1/2} R_i(-\lambda) \Phi(\Sigma) R_i(-\lambda) \Sigma^{1/2}\big)\big|
    & \leq 
    \|\Sigma^{1/2} R_i(-\lambda) \Phi(\Sigma) R_i(-\lambda) \Sigma^{1/2}\|_2
    \nonumber 
    \\
    & \leq 
    \|\Sigma^{1/2}\|_2^2 \|\Phi(\Sigma)\|_2 \|R_i(-\lambda)\|_2^2
    \nonumber 
    \\
    & \leq 
    \frac{\|\Sigma^{1/2}\|_2^2 \|\Phi(\Sigma)\|_2 }{\lambda^2}
    \nonumber 
    \\
    & = 
    \frac{\|\Sigma\|_2 \|\Phi(\Sigma)\|_2 }{\lambda^2} 
    \label{equ:MaxEigenVal}
\end{align}
Combining these two facts yields the upper bound 
\begin{align*}
    |\delta_2| \leq 
    \frac{2 \|\Sigma\|^2_2 \|\Phi(\Sigma)\|_2}{\lambda^3 n}
\end{align*}
which goes to zero as $n \rightarrow \infty$.

We now proceed to bound $\delta_3$ and $\delta_4$. 
With the bound on the trace \eqref{equ:MaxEigenVal} as well as using the bound $|(1+a)^{-2} - (1+b)^{-2}| \leq 2 |b-a|$ we arrive at the bound for $\delta_3$
\begin{align*}
    & |\delta_3|
    \leq 
    2 \frac{\|\Sigma\|_2 \|\Phi(\Sigma)\|_{2}}{\lambda^2}\\
    & \quad \times \max_{1 \leq i \leq n} 
    \Big|Z_{i} \Sigma^{1/2}R_i(z)\Phi(\Sigma)R_i(z)\Sigma^{1/2}Z_i^{\top} - \trace\big(\Sigma^{1/2} R_i(z) \Phi(\Sigma) R_i(z) \Sigma^{1/2}\big)\Big|.
\end{align*}
Meanwhile using that $1 + \frac{1}{n} Z_i \Sigma^{1/2} R_{i}(-\lambda) \Sigma^{1/2} Z_{i}^{\top} \geq 1$ we arrive at the bound for $\delta_4$ 
\begin{align*}
    |\delta_4|
    \leq 
    \max_{1 \leq i \leq n} 
    \Big|Z_{i} \Sigma^{1/2}R_i(z)\Phi(\Sigma)R_i(z)\Sigma^{1/2}Z_i^{\top} - \trace\big(\Sigma^{1/2} R_i(z) \Phi(\Sigma) R_i(z) \Sigma^{1/2}\big)\Big|
\end{align*}
We now show that $\max_{1 \leq i \leq n} \Big|Z_{i} \Sigma^{1/2}R_i(z)\Phi(\Sigma)R_i(z)\Sigma^{1/2}Z_i^{\top} - \trace\big(\Sigma^{1/2} R_i(z) \Phi(\Sigma) R_i(z) \Sigma^{1/2}\big)\Big|$ converges to zero almost surely. 
Observe since we have the upper bound on the largest eigenvalue we have using Lemma \ref{lem:Conc} as well as union bound for $1 \leq i \leq n$ we have for $0 < t < \frac{\|\Sigma\|_2 \|\Phi(\Sigma)\|_2 }{\lambda^2}  $
\begin{align}
\nonumber
    & \P\Big( \max_{1 \leq i \leq n}
    \frac{1}{d} 
    \Big|
    Z_{i} \Sigma^{1/2}R_i(z)\Phi(\Sigma)R_i(z)\Sigma^{1/2}Z_i^{\top} 
    - 
    \trace\big(\Sigma^{1/2} R_i(z) \Phi(\Sigma) R_i(z) \Sigma^{1/2}\big)
    \Big|
    \geq t 
    \Big)\\
    & \leq 
    2 
    \exp\Big\{ - \frac{ d t^2 \lambda^4 }{6 \|\Sigma\|_2^2 \|\Phi(\Sigma)\|^2 } + \log(n)\Big\}
    \label{equ:ConcApplied}
\end{align}
Let $V_{n,d} := \max_{1 \leq i \leq n}
    \frac{1}{d} 
    \Big|
    Z_{i} \Sigma^{1/2}R_i(z)\Phi(\Sigma)R_i(z)\Sigma^{1/2}Z_i^{\top} 
    - 
    \trace\big(\Sigma^{1/2} R_i(z) \Phi(\Sigma) R_i(z) \Sigma^{1/2}\big)
    \Big|
    $ and, for any $t > 0$, let $E_{n,d} (t)$ denote the event $\{ V_{n,d} \geq t \}$ where $d = d_n$.
    Then, if $d = d_n$ satisfies $d_n / n \to \infty$, $\P (E_{n,d}) \leq 2 n \exp\Big\{ - \frac{ d t^2 \lambda^4 }{6 \|\Sigma\|_2^2 \|\Phi(\Sigma)\|^2 } \Big\} \leq 2 n \exp\Big\{ - \frac{\gamma n t^2 \lambda^4 }{12 \|\Sigma\|_2^2 \|\Phi(\Sigma)\|^2 } \Big\}$ where the last inequality for $n$ large enough that $d/n \geq \gamma / 2$.
    Hence,
    \begin{equation*}
      \sum_{n=1}^{\infty} \P ( E_{n,d_n} (t)) < + \infty
    \end{equation*}
    so that, by the Borel-Cantelli lemma, almost surely, $V_{n,d} \geq t$ only holds for a finite number of values of $n$.
    This implies that, almost surely, $\limsup_{n \to \infty} V_{n,d} \leq t$.
    Note that this is true for every $t > 0$; letting $t=1/k$ and taking a union bound over $k \geq 1$ shows that $\limsup_{n \to \infty} V_{n,d} = 0$ almost surely, \ie $V_{n,d} \to 0$ almost surely.

\end{document}